\documentclass[oneside,english]{amsart}
\usepackage[T1]{fontenc}
\usepackage[latin9]{inputenc}
\usepackage{amsthm}
\usepackage{amstext}
\usepackage{amssymb}
\usepackage{esint}
\usepackage{amscd}
\usepackage{color}
\makeatletter
\newtheorem*{theoremA}{Theorem A}
\newtheorem*{theoremB}{Theorem B}
\numberwithin{equation}{section}
\numberwithin{figure}{section}
\theoremstyle{plain}
\newtheorem{theorem}{\protect\theoremname}[section]
\newtheorem{proposition}[theorem]{\protect\propositionname}
\theoremstyle{plain}
\theoremstyle{definition}
\newtheorem{definition}[theorem]{\protect\definitionname}
\theoremstyle{plain}
\newtheorem{lemma}[theorem]{\protect\lemmaname}
\newtheorem{corollary}[theorem]{\protect\corollaryname}
\theoremstyle{plain}
\newtheorem{remark}[theorem]{\protect\remarkname}
\theoremstyle{plain}
\makeatother

\usepackage{babel}
\providecommand{\definitionname}{Definition}
\providecommand{\lemmaname}{Lemma}
\providecommand{\theoremname}{Theorem}
\providecommand{\corollaryname}{Corollary}
\providecommand{\remarkname}{Remark}
\providecommand{\propositionname}{Proposition}
	
\DeclareMathOperator{\loc}{loc}

\DeclareMathOperator{\dist}{dist}

\DeclareMathOperator{\R}{\mathbb{R}}

\begin{document}

\title[Neumann-to-Steklov limits]{Weighted Neumann-to-Steklov limits for nonlinear eigenvalues and trace constants}

\author{Alexander Menovschikov}
\thanks{This work was supported by the Russian Science Foundation under grant No~.25-71-00064, https://rscf.ru/en/project/25-71-00064/} 

\begin{abstract}
We study a nonlinear Neumann-to-Steklov limit generated by a family of
interior weights concentrating at the boundary.  On a class of admissible
possibly irregular domains obtained from the unit ball by trace-compatible
Sobolev homeomorphisms, we consider the first nontrivial weighted
\((p,q)\)-Neumann eigenvalue with respect to a concentrating bulk weight
\(\gamma_a\).  We prove that, as \(a\to0\), these eigenvalues converge to the
corresponding weighted \((p,q)\)-Steklov eigenvalue with boundary weight
\(\beta\).  Moreover, normalized minimizers converge, up to subsequences,
strongly in \(W^{1,p}\) to Steklov minimizers.  Equivalently, the best constants
in the weighted Poincar\'e inequalities converge to the best constants in the
weighted trace inequalities; in fact, a quantitative convergence estimate is
obtained in the subcritical trace range.
\end{abstract}
\maketitle


\section{Introduction}

The Steklov problem and its variational counterparts occupy a central place in
spectral geometry and in the theory of trace inequalities; for general background
we refer to the surveys of Girouard--Polterovich~\cite{GirouardPolterovich2017}
and Colbois--Girouard--Gordon--Sher~\cite{ColboisGirouardGordonSher2024}. In
this paper we study a nonlinear Neumann-to-Steklov limit generated by a family of
interior weights concentrating at the boundary. More precisely, we consider the
first nontrivial weighted \((p,q)\)-Neumann eigenvalue \(\Lambda^N_{p,q}(\gamma_a)\) for the \(p\)-Laplace operator \(\Delta_p\) associated with a
bulk weight \(\gamma_a\), where the parameter \(a>0\) controls the concentration
of mass near \(\partial\Omega\). As \(a\to0\), the corresponding weighted measures
accumulate on the boundary, and one naturally expects the Neumann problem to
approach a weighted \((p,q)\)-Steklov problem for the \(p\)-Laplace and eigenvalues \(\Lambda^N_{p,q}(\gamma_a)\) converge to the first non-trivial Steklov eigenvalue \(\Lambda^{St}_{p,q}(\beta)\). Our goal is
to justify this limit in a weighted nonlinear setting and to express it
equivalently in terms of the first nonlinear eigenvalues and of the best
constants \(C^N_{p,q}(\gamma_a)\) and \(C^{St}_{p,q}(\beta)\) in the corresponding weighted Poincar\'e and trace inequalities.

Several neighboring lines of research clarify different aspects of this picture,
but mostly in settings that remain essentially disjoint. In the linear smooth
case, direct Neumann-to-Steklov limits generated by concentration of mass near
the boundary were studied by Lamberti and Provenzano~\cite{LambertiProvenzano2017}
and, at a finer asymptotic level, by Dalla Riva and
Provenzano~\cite{DallaRivaProvenzano2018}. On the variational side, the
convergence of bulk Sobolev-type constants to trace constants was investigated,
in a form particularly close in spirit to the present work, by Arrieta,
Rodr\'iguez-Bernal, and Rossi~\cite{ArrietaRodriguezBernalRossi2008}; see also
\cite{BonderRossi2002,BonderRossiFerreira2003}. A conceptually related
measure-theoretic framework for the linear Laplace operator and Steklov eigenvalues, based on Radon
measures and continuity of variational spectra, was developed by Girouard,
Karpukhin, and Lagac\'e~\cite{GirouardKarpukhinLagace2021}. On the other hand,
for singular domains and Jacobian-induced weights, the natural analytic language
is that of composition operators on Sobolev spaces and induced weighted
embeddings, as in the work of Gol'dshtein and
Ukhlov~\cite{GoldshteinUkhlov2009}; in this direction, weighted Neumann and
weighted Steklov problems on outward cuspidal domains have recently been studied
in~\cite{GarainGoldshteinUkhlov2025,MenovschikovUkhlov2026,LambertiUkhlov2026}.
The main point of the present paper is to connect these themes. We establish a nonlinear weighted Neumann-to-Steklov limit for
Jacobian-induced concentrating measures on admissible possibly singular domains,
together with quantitative convergence of the corresponding sharp constants and
strong convergence of minimizers.

Our setting is based on a transfer-from-the-ball construction. Namely, we assume
that \(\Omega\) is the image of the unit ball under a trace-compatible Sobolev
homeomorphism \(\varphi:\overline B\to\overline\Omega\), so that the
corresponding composition operator transfers Sobolev and trace information from
the model domain \(B\) to \(\Omega\). Starting from a bounded boundary weight
\(\alpha\) on \(\partial B\), we define a family of radial concentrating
densities
\[
\rho_a(x)=\frac{n}{a}|x|^{n/a-n},
\qquad x\in B,
\]
and then transport them to \(\Omega\) through \(\varphi\), which produces the
bulk weights \(\gamma_a\) in \(\Omega\) and the limiting boundary weight
\(\beta\) on \(\partial\Omega\). This framework includes, in particular,
bilipschitz images of the ball and standard outward cuspidal domains, while
allowing the geometry of \(\Omega\) to enter the problem through Jacobian
factors rather than through classical smoothness assumptions on \(\partial\Omega\).

Our main results concern quantitative convergence of the corresponding sharp
constants and variational levels, as well as compactness and strong convergence
of normalized minimizers. The exact statements are provided in Section~5.

\begin{theoremA}
\label{thm:quantitative-neumann-to-steklov-intro}
Let \(1<p<n\), \(1<q<\frac{p(n-1)}{n-p}\), and let
\(\Omega\subset\mathbb R^n\) be admissible.
Then for every \(s\in(0,\delta_{p,q})\) there exists a constant \(C_s>0\) such that
\[
\left|
C^N_{p,q}(\gamma_a)-C^{St}_{p,q}(\beta)
\right|
\le
C_s a^s
\qquad\text{for every }a\in(0,1].
\]
Consequently,
\[
C^N_{p,q}(\gamma_a)\to C^{St}_{p,q}(\beta)
\qquad\text{as }a\to0,
\]
In particular,
\[
\left|
\Lambda^N_{p,q}(\gamma_a)-\Lambda^{St}_{p,q}(\beta)
\right|
\le
C_s^{\,\Lambda} a^s, \qquad
\Lambda^N_{p,q}(\gamma_a)\to \Lambda^{St}_{p,q}(\beta)
\quad\text{as }a\to0.
\]
\end{theoremA}

\begin{theoremB}
\label{thm:convergence-of-minimizers-intro}
Let \(1<p<n\), \(1<q<\frac{p(n-1)}{n-p}\), and let
\(\Omega\subset\mathbb R^n\) be admissible. Let \(a_k\to0\), and let
\(u_k\) be normalized minimizers for
\(\Lambda^N_{p,q}(\gamma_{a_k})\). Then, after passing to a subsequence,
\[
u_k\to u
\qquad\text{strongly in }W^{1,p}(\Omega),
\]
where \(u\) is a normalized minimizer for
\(\Lambda^{St}_{p,q}(\beta)\).
\end{theoremB}

The proof proceeds in two stages. We first analyze a model concentration problem
on the unit ball, where we combine a quantitative comparison between the Poisson
extension and the radial extension of the trace, a boundary-layer estimate for
zero-trace functions, and the convergence of weighted bulk moments to the
corresponding boundary moments. We then transfer these estimates to the general
domain \(\Omega\) by means of the map \(\varphi\).
This yields a quantitative comparison of the bulk and boundary quotient
seminorms, from which the convergence of sharp constants and eigenvalues follows,
while the convergence of minimizers is obtained by combining the same
concentration mechanism with compactness of the induced trace operator. 

The paper is organized as follows. In Section~2 we introduce the analytic and
geometric framework, including admissible transfer mappings, admissible domains,
and the induced bulk and boundary weights. Section~3 contains the variational
formulation of the weighted Neumann and weighted Steklov problems, the existence
of the first nontrivial eigenvalues, and their interpretation in terms of
sharp weighted Poincar\'e and trace constants. In Section~4 we establish the
boundary concentration mechanism, first on the unit ball and then on general
admissible domains via the transfer map \(\varphi\). Finally, Section~5 is
devoted to the proofs of the main results and the unweighted boundary limit as a particular case.

\section{Analytic and geometric framework}

Throughout the paper, \(B=B(0,1)\subset \mathbb R^n\) denotes the unit ball,
\(\partial B=S^{n-1}\), and \(\Omega\subset \mathbb R^n\) is a bounded domain. 
As a standard assumption on the regularity of $\Omega$, we will assume that its boundary is locally \((n-1)\)-rectifiable.
\begin{definition}
A Borel set \(E\subset\mathbb R^n\) is called locally \((n-1)\)-rectifiable if
\(E\) is countably \((n-1)\)-rectifiable and
\[
\mathcal H^{n-1}(E\cap K)<\infty
\qquad\text{for every compact }K\subset\mathbb R^n.
\]
Equivalently, there exist Lipschitz maps
\[
F_j:\mathbb R^{n-1}\to\mathbb R^n,\qquad j=1,2,\dots,
\]
such that
\[
\mathcal H^{n-1}\!\left(E\setminus\bigcup_{j=1}^\infty F_j(\mathbb R^{n-1})\right)=0,
\]
and \(\mathcal H^{n-1}(E\cap K)<\infty\) for every compact \(K\subset\mathbb R^n\).
\end{definition}

We write \(d\sigma\) for the surface measure on \(\partial B\) and \(dS := d\mathcal H^{n-1}\!\lfloor_{\partial\Omega}\) for the
$n-1$ Hausdorff measure on \(\partial\Omega\).

\subsection{Function spaces}

We begin by fixing the basic terminology concerning function spaces.

Let \((X,\mu)\) be a measure space. A \emph{weight} on \(X\) is a nonnegative
measurable function \(w:X\to[0,\infty)\). When \(X\) is a locally compact
measure space, we shall usually assume \(w\in L^1_{\mathrm{loc}}(X,\mu)\).

For \(1\le p<\infty\), the \textit{weighted Lebesgue space} \(L^p(X,w)\) consists of all measurable functions \(f\) such that
\[
\|f\|_{L^p(X,w)} := \left(\int_X |f|^p w\,d\mu\right)^{1/p}
<\infty .
\]
As usual, functions are identified if they agree almost everywhere with respect to the weighted measure $w\,d\mu$. With the above norm this is a Banach space, see
Brezis~\cite[Theorem~4.8]{Brezis}.

\begin{lemma}[Bounded weights preserve continuous and compact embeddings]
\label{lem:bounded-weight-embedding}
Let \(U\subset\mathbb R^n\) be a bounded measurable set, let \(1\le p<\infty\),
and let \(w\in L^\infty(U)\), \(w\ge0\). If a Banach space \(Y\) is continuously
(resp. compactly) embedded into \(L^p(U)\), then \(Y\) is also continuously
(resp. compactly) embedded into \(L^p(U,w)\).

The same statement holds for \(U=\partial\Omega\) with respect to the surface
measure.
\end{lemma}

\begin{proof}
The claim follows from the estimate
\[
\|f\|_{L^p(U,w)}
\le
\|w\|_{L^\infty(U)}^{1/p}\|f\|_{L^p(U)}.
\]
\end{proof}

Let \(U\subset\mathbb R^n\) be an open set and let \(1<p<\infty\). The \textit{Sobolev
space} \(W^{1,p}(U)\) consists of all functions \(u\in L^p(U)\) whose weak first
order derivatives belong to \(L^p(U)\). It is equipped with the norm
\[
\|u\|_{W^{1,p}(U)}
:=
\|u\|_{L^p(U)}+\|\nabla u\|_{L^p(U)} .
\]
We denote by \(W^{1,p}_0(U)\) the closure of \(C^\infty_0(U)\) in
\(W^{1,p}(U)\).

In the model domain \(B=B(0,1)\), we shall use the standard Sobolev, Poincare,
Hardy, compactness and trace theorems (see, e.g. \cite{Brezis}). On a general domain \(\Omega\), whose
boundary will only be assumed to satisfy the structural hypotheses introduced
below, the corresponding properties will either be imposed explicitly or derived
from the composition-operator assumptions. 

\begin{proposition}[Sobolev inequality on the ball]\label{prop:sobolev-ball}
Let \(1<p<n\), and let
\[
p^*:=\frac{np}{n-p}.
\]
Then there exists a constant \(C=C(n,p)>0\) such that
\[
\|u\|_{L^{p^*}(B)}
\le
C\|\nabla u\|_{L^p(B)}
\qquad\text{for all }u\in W^{1,p}_0(B).
\]
\end{proposition}

\begin{proposition}[Poincare inequality on the ball]\label{prop:poincare-ball}
Let \(1<p<\infty\). Then there exists a constant \(C=C(n,p)>0\) such that
\[
\left\|u-\frac{1}{|B|}\int_B u\,dx\right\|_{L^p(B)}
\le
C\|\nabla u\|_{L^p(B)}
\qquad\text{for all }u\in W^{1,p}(B).
\]
\end{proposition}

\begin{proposition}[Rellich--Kondrachov theorem on the ball]\label{prop:rellich-ball}
Let \(1<p<n\) and \(1\le r<p^*\), where
\[
p^*=\frac{np}{n-p}.
\]
Then the embedding
\[
W^{1,p}(B)\hookrightarrow L^r(B)
\]
is compact.
\end{proposition}

\begin{proposition}[Hardy inequality on the ball]\label{prop:hardy-ball}
Let \(1<p<n\), and let
\[
d(x):=\operatorname{dist}(x,\partial B)=1-|x|,
\qquad x\in B.
\]
Then there exists a constant \(C=C(n,p)>0\) such that
\[
\int_B \frac{|u(x)|^p}{d(x)^p}\,dx
\le
C\int_B |\nabla u(x)|^p\,dx
\qquad\text{for all }u\in W^{1,p}_0(B).
\]
\end{proposition}

For a general domain \(U\subset\mathbb R^n\), we define the \emph{abstract trace space}
of \(W^{1,p}(U)\) by
\[
Tr(W^{1,p}(U))
:=
W^{1,p}(U)/W^{1,p}_0(U),
\]
equipped with the quotient norm
\[
\|[u]\|_{Tr(W^{1,p}(U))}
:=
\inf\Bigl\{
\|u+v\|_{W^{1,p}(U)}:\ v\in W^{1,p}_0(U)
\Bigr\}.
\]
Here \([u]\) denotes the coset of \(u\in W^{1,p}(U)\) modulo \(W^{1,p}_0(U)\).

Whenever a bounded linear trace operator
\[
T_U:W^{1,p}(U)\to X(\partial U)
\]
is available and satisfies
\[
\ker T_U=W^{1,p}_0(U),
\]
it induces a canonical isomorphism between \(Tr(W^{1,p}(U))\) and the image \(T_U(W^{1,p}(U))\).

If \(\Omega\subset\mathbb R^n\) is a bounded Lipschitz domain, then the classical trace
theory (see, for instance, Grisvard~\cite[Theorem~1.5.1.2 and Corollary~1.5.1.6]{Grisvard}) yields a bounded surjective trace operator
\[
T:W^{1,p}(\Omega)\to W^{1-1/p,p}(\partial \Omega),
\]
and the abstract trace space \(Tr(W^{1,p}(\Omega))\) can be canonically
identified with the fractional Sobolev space \(W^{1-1/p,p}(\partial \Omega)\).

For completeness, recall that if \(\Gamma\subset\mathbb R^n\) is a compact Lipschitz
hypersurface, \(0<s<1\), and \(1\le p<\infty\), then \(W^{s,p}(\Gamma)\) may be described
via the Slobodeckij norm
\[
\|g\|_{W^{s,p}(\Gamma)}^p
=
\|g\|_{L^p(\Gamma)}^p
+
\iint_{\Gamma\times\Gamma}
\frac{|g(\xi)-g(\eta)|^p}{|\xi-\eta|^{(n-1)+sp}}
\,dS(\xi)\,dS(\eta).
\]

\begin{proposition}[Trace theorem on the ball]
Let \(1<p<\infty\). Then the trace operator
\[
T:W^{1,p}(B)\to W^{1-1/p,p}(\partial B)
\]
is bounded and surjective.
\end{proposition}

The next two statements are standard consequences of the fractional Sobolev
embedding and compactness theorems, applied in local charts on the compact
smooth manifold \(\partial B\); see, for example, Di Nezza--Palatucci--Valdinoci
\cite{DNPV}.

\begin{proposition}[Compactness of the trace embedding on the sphere]\label{compsph}
Let \(1<p<n\) and
\[
1\le q<\frac{p(n-1)}{n-p}.
\]
Then the embedding
\[
W^{1-1/p,p}(\partial B)\hookrightarrow L^q(\partial B)
\]
is compact.
\end{proposition}

\begin{proposition}[Fractional embeddings on the sphere]
\label{prop:fractional-sphere-embedding}
Let \(1<p<n\), \(1<q<p(n-1)/(n-p)\), and let
\[
\delta_{p,q}
:=
\min\left\{
1-\frac1p,\,
1-\frac{n}{p}+\frac{n-1}{q}
\right\}.
\]
Then, for every \(0<s<\delta_{p,q}\), the embedding
\[
W^{1-1/p,p}(\partial B)\hookrightarrow W^{s,q}(\partial B)
\]
is continuous.
\end{proposition}

The particular exponent \(\delta_{p,q}\) is chosen for later use: it is the
largest regularity exponent needed simultaneously in the boundary-layer
estimate and in the Poisson approximation argument.

\subsection{Change of variables in the bulk and on the boundary}

We recall the change-of-variables formulae used below.  In the bulk we shall
use the Sobolev change-of-variables formula, while on the boundary we shall use
the density of the image surface measure, interpreted under a doubling
assumption as a volume derivative.  For the bulk formula we refer to
Federer~\cite{Federer} and Haj{\l}asz~\cite{Hajlasz1993}; for the set-function
and volume-derivative point of view, see Vodop'yanov--Ukhlov
\cite{VodopyanovUkhlov2004,VodopyanovUkhlov2005}.

Let \(\Omega,\widetilde\Omega\subset\mathbb R^n\) be domains. We say that a mapping
\(\varphi:\Omega\to \widetilde\Omega\) belongs to \(W^{1,p}_{\mathrm{loc}}(\Omega;\mathbb R^n)\),
\(1\le p\le\infty\), if its coordinate functions belong to
\(W^{1,p}_{\mathrm{loc}}(\Omega)\). In this case the weak differential
\(D\varphi(x)\) is defined for a.e. \(x\in \Omega\), and we write
\[
J_\varphi(x):=\det D\varphi(x).
\]
We denote by \(|D\varphi(x)|\) the operator norm of the matrix \(D\varphi(x)\).

Let us recall the change of variables formula for homeomorphisms $\varphi:B \to\Omega$. It can be found, for example, in \cite{Halmos1950}. We define a volume derivative of the inverse mapping by
$$
J_{\varphi^{-1}}(y) := \lim_{r \to 0} \frac{|\varphi^{-1}(B(y,r))|}{|B(y,r)|},
$$
where $B(y,r)$ is a ball with a center at a point $x$ and with a radius $r$.
This function $J_{\varphi^{-1}}$ belongs to the space $L_{1,\loc}(\Omega)$ and, if $\varphi^{-1} \in W^1_{n,\loc}(\Omega)$, then this volume derivative coincides with the Jacobian a.e. in $\Omega$: $J_{\varphi^{-1}}(y) = \det D\varphi(x)$ for almost all $y\in\Omega$ (see, for example, \cite{Hajlasz1993,VGR1979}).

Let, in addition, the homeomorphism $\varphi : \Omega\to \widetilde\Omega$ possess the Luzin $N^{-1}$-property (the preimage of a set of measure zero has measure zero), then the following change of variables formula 
\begin{equation}
\label{chvf2}
\int\limits_{B}f\circ\varphi (x) dx=\int\limits_{\Omega} f(y) J_{\varphi^{-1}}(y) dy,
\end{equation}
holds for every nonnegative measurable function $f:\Omega\to\mathbb R$.

As before, consider a bounded domain \(\Omega\subset \mathbb R^n\) with locally \((n-1)\)-rectifiable boundary. Now assume that our $\varphi:B \to \Omega$ can be homeomorphicaly extended to the boundary 
$$
    \varphi: \overline{B} \to \overline{\Omega}.
$$
Consider the boundary homeomorphism
$$
    \varphi_\partial:=\varphi|_{\partial B}: \partial B \to \partial\Omega.
$$ 

To pass from the bulk to the boundary, one can no longer rely on the ordinary
Lebesgue differentiation theory in \(\mathbb R^n\). Instead, we use the
set-function approach in metric measure spaces of homogeneous type, which
provides the appropriate notion of volume derivative for the inverse boundary
mapping; for general definitions and results, see
Vodop'yanov--Ukhlov~\cite{VodopyanovUkhlov2004}. In order to work in this
framework, we impose in addition the following doubling assumption on
\(\partial\Omega\).

We say that the boundary $\partial\Omega$ satisfies the \textit{doubling condition}, if there exists a  constant $C$ such that
$$
     \mathcal{H}^{n-1}(\partial\Omega \cap B^\partial(x,2r)) \leq C\mathcal{H}^{n-1}(\partial\Omega \cap B^\partial(x,r)),
$$
where $B^\partial$ is an Eucledian ball on the boundary.

\begin{proposition}\cite{VodopyanovUkhlov2004}\label{cov-bound}
    Let $\Omega \subset \R^n$ be a domain with locally \((n-1)\)-rectifiable boundary $\partial\Omega$ satisfying the doubling condition. Assume that the boundary homeomorphism $\varphi_\partial: \partial B \to \partial\Omega$ satisfy \textit{Luzin $N^{-1}$-property}:
    $$
    \mathcal{H}^{n-1}(E)=0 \quad\Longrightarrow \quad \mathcal{H}^{n-1}(\varphi_\partial^{-1}(E))=0, \quad E \in \partial\Omega.
    $$ 
    
    Then for almost all $t \in \partial\Omega$ there exists a nonnegative measurable function $J^\partial_{\varphi^{-1}}: \partial\Omega \to [0,\infty)$ (the \emph{volume derivative of the inverse boundary mapping}),
    $$
        J^\partial_{\varphi^{-1}} (t) = \lim_{r\to0}
\frac{\sigma\bigl(\psi_\partial(\partial\Omega\cap B(t,r))\bigr)}
     {\mathcal H^{n-1}(\partial\Omega\cap B(t,r))},
    $$
    such that 
\begin{equation}\label{eq:weak-cov-inverse}
\int_{\partial\Omega} f(t)\,J^\partial_{\varphi^{-1}}(t)\,dS
=
\int_{\partial B} f(\varphi_\partial(s))\,d\sigma
\end{equation}
for every nonnegative measurable function \(f:\partial\Omega\to[0,\infty]\)
\end{proposition}

\subsection{Composition operators and induced compact trace embeddings}

The main tool for transferring problems from the reference domain \(B\) to a
general domain \(\Omega\) is the theory of composition operators on Sobolev
spaces. This theory has been developed over several decades; in the present
context we refer, in particular, to the early work of Ukhlov and
Vodop'yanov--Ukhlov, to the weighted embedding theory of Gol'dshtein--Ukhlov,
and to later applications to spectral and weighted Poincar\'e--Sobolev problems
on singular domains; see, for instance,
\cite{Ukhlov1993, VodopyanovUkhlov1998, GoldshteinUkhlov2009, GoldshteinUkhlov2017, Vodopyanov2024}.

Recall that a homeomorphism $\varphi: B\to\Omega$ is called the \textit{weak $p$-quasiconformal mapping} \cite{GoldshteinGurovRomanov1995, VodopyanovUkhlov1998}, if $\varphi\in W^{1,p}_{loc}(B)$ has finite distortion, i.e. 
$$
    D\varphi(x) = 0 \,\,\,\text{for a.e. on the set } \,\,\{x\in B:J_\varphi(x)=0\}
$$
and 
$$
K_p(\varphi)=\operatorname{esssup}_{B}\frac{|D\varphi(x)|}{|J_\varphi(x)|^{1/p}}<\infty,\,\,1\leq p < \infty.
$$
Note that the weak $p$-quasiconformal mapping $\varphi: B\to\Omega$ has the Luzin $N^{-1}$-property \cite{VodopyanovUkhlov1998} and its Jacobian $J_\varphi>0$ a.e. in $B$.

Let us formulate the following theorem from \cite{VodopyanovUkhlov1998}, adapted to our settings:

\begin{theorem}
\label{comp_w}
Let $\varphi: B\to\Omega$ be a weak $p$-quasiconformal mapping, $1<p<n$, and
$$
J_\varphi \in L^{n/p}(B).
$$ 
Then $\varphi$ generates a bounded composition operator
\[
\varphi^{\ast}:W^{1,p}(\Omega)\to W^{1,p}(B)
\]
by the composition rule $\varphi^{\ast}(u)=u\circ\varphi$.
\end{theorem}

We next transfer compact Sobolev embeddings from \(B\) to \(\Omega\) by means of
the volume derivative \(J_{\varphi^{-1}}\) of the inverse mapping.

\begin{theorem}[Weighted embedding and Poincar\'e inequality induced by \(\varphi\)]
\label{thm:bulk-jacobian}
Let $\varphi: B\to\Omega$ be a weak $p$-quasiconformal mapping, $1<p<n$, $1\le q<\frac{np}{n-p}$.

Then the embedding
\[
W^{1,p}(\Omega)\hookrightarrow L^q(\Omega,J_{\varphi^{-1}})
\]
is compact and \(\Omega\) is an \(J_{\varphi^{-1}}\)-weighted
\((p,q)\)-Poincar\'e domain, i.e. there exists a constant \(C=C(\varphi,p,q)>0\)
such that
\[
\inf_{c\in\mathbb R}\|u-c\|_{L^q(\Omega,J_{\varphi^{-1}})} \le C\,\|\nabla u\|_{L^p(\Omega)} \qquad\text{for all }u\in W^{1,p}(\Omega).
\]
\end{theorem}

\begin{proof}
By the change-of-variables formula \eqref{chvf2},
\[
\|u\|_{L^q(\Omega,J_{\varphi^{-1}})}=\|u\circ\varphi\|_{L^q(B)}.
\]

Let \(\{u_k\}\) be bounded in \(W^{1,p}(\Omega)\). By Theorem~\ref{comp_w}, the sequence \(v_k:=u_k\circ\varphi\)
is bounded in \(W^{1,p}(B)\). Since \(1\le q<\frac{np}{n-p}\),
the Rellich--Kondrachov theorem on the ball implies that, after passing to a
subsequence,
\[
v_k\to v
\qquad\text{strongly in }L^q(B).
\]
Therefore,
\[
\|u_k-u_m\|_{L^q(\Omega,J_{\varphi^{-1}})}
=
\|v_k-v_m\|_{L^q(B)}\to 0
\qquad (k,m\to\infty),
\]
so \(\{u_k\}\) is Cauchy in \(L^q(\Omega,J_{\varphi^{-1}})\). Since \(L^q(\Omega,J_{\varphi^{-1}})\) is complete,
\(\{u_k\}\) converges strongly in \(L^q(\Omega,J_{\varphi^{-1}})\). Thus the embedding is compact.

To prove Poincar\'e inequality, again by the change-of-variables formula,
\[
\inf_{c\in\mathbb R}\|u-c\|_{L^q(\Omega,J_{\varphi^{-1}})}
=
\inf_{c\in\mathbb R}\|u\circ\varphi-c\|_{L^q(B)}.
\]
By the classical \((p,q)\)-Poincar\'e inequality on the ball,
\[
\inf_{c\in\mathbb R}\|u\circ\varphi-c\|_{L^q(B)}
\le
C_{p,q}(B)\,\|\nabla(u\circ\varphi)\|_{L^p(B)}.
\]
Since $\varphi$ is a measurable weak \(p\)-quasiconformal mapping, we have
\[
\|\nabla(u\circ\varphi)\|_{L^p(B)}
\le
K_p(\varphi)\,\|\nabla u\|_{L^p(\Omega)}.
\]
Combining the above estimates, we obtain
\[
\inf_{c\in\mathbb R}\|u-c\|_{L^q(\Omega,J_{\varphi^{-1}})}
\le
C_{p,q}(B)\,K_p(\varphi)\,\|\nabla u\|_{L^p(\Omega)},
\]
which proves the weighted \((p,q)\)-Poincar\'e inequality.
\end{proof}

\begin{corollary}[Compact embedding into \(L^p(\Omega)\)]
\label{cor:compact-Lp}
Let \(1<p<n\), and let \(\varphi:B\to\Omega\)
be a weak \(p\)-quasiconformal mapping. Assume in addition that
\[
J_\varphi\in L^r(B)
\qquad\text{for some } r>\frac{n}{p}.
\]
Then the embedding
\[
W^{1,p}(\Omega)\hookrightarrow L^p(\Omega)
\]
is compact.
\end{corollary}

\begin{proof}
Let \(\{u_k\}\) be a bounded sequence in \(W^{1,p}(\Omega)\), and set \( v_k:=u_k\circ\varphi\).
By Theorem~\ref{comp_w}, the sequence \(\{v_k\}\) is bounded in \(W^{1,p}(B)\).

Let
\[
r'=\frac{r}{r-1}.
\]
Since \(r>\frac{n}{p}\), we have
\[
pr'=\frac{pr}{r-1}<\frac{np}{n-p}=:p^*.
\]
Therefore, by the Rellich--Kondrachov theorem on the ball, after passing to a
subsequence we may assume that
\[
v_k\to v
\qquad\text{strongly in }L^{pr'}(B).
\]

Now let \(k,m\in\mathbb N\). By the change-of-variables formula,
\[
\|u_k-u_m\|_{L^p(\Omega)}^p
=
\int_\Omega |u_k(y)-u_m(y)|^p\,dy
=
\int_B |v_k(x)-v_m(x)|^p J_\varphi(x)\,dx.
\]
Applying H\"older's inequality with exponents \(r\) and \(r'\), we obtain
\[
\|u_k-u_m\|_{L^p(\Omega)}^p
\le
\|J_\varphi\|_{L^r(B)}
\left(
\int_B |v_k-v_m|^{pr'}\,dx
\right)^{1/r'}.
\]
Since \(\{v_k\}\) converges strongly in \(L^{pr'}(B)\), it is a Cauchy sequence in
that space, and hence \(\{u_k\}\) is Cauchy in \(L^p(\Omega)\).

Thus every bounded sequence in \(W^{1,p}(\Omega)\) admits a subsequence convergent
in \(L^p(\Omega)\), which proves the compactness of the embedding
\[
W^{1,p}(\Omega)\hookrightarrow L^p(\Omega).
\]
\end{proof}

Having obtained the bounded composition operator in the bulk, we now pass to the
boundary. The trace on \(\Omega\) will be induced from the classical trace on the
ball via the boundary homeomorphism \(\varphi_\partial\). The corresponding target
space on \(\partial\Omega\) is a weighted Lebesgue space determined by the
boundary volume derivative of \(\varphi_\partial^{-1}\).

\begin{theorem}[Induced compact weighted trace operator]\label{thm:induced-trace-operator}
Let $\Omega \subset \R^n$ be a domain with locally \((n-1)\)-rectifiable boundary $\partial\Omega$ satisfying the doubling condition. Let $\varphi: \overline{B}\to\overline\Omega$ be a weak $p$-quasiconformal mapping and the boundary homeomorphism $\varphi_\partial: \partial B \to \partial\Omega$ satisfy \textit{Luzin $N^{-1}$-property}, \(1<p<n\),\( 1\le q<\frac{p(n-1)}{n-p}\)

Then there exists a compact operator
\[
T:W^{1,p}(\Omega)\to L^q(\partial\Omega,J^\partial_{\varphi^{-1}})
\]
such that for every \(u\in W^{1,p}(\Omega)\),
\[
Tu = \bigl(T_B(u\circ\varphi)\bigr)\circ\varphi^{-1}_\partial \qquad \text{in }L^q(\partial\Omega,J^\partial_{\varphi^{-1}}).
\]
\end{theorem}

\begin{proof}
By Theorem~\ref{comp_w}, the composition operator
\[
\varphi^*:W^{1,p}(\Omega)\to W^{1,p}(B)
\]
is bounded.
By the trace theorem on the ball,
\[
T:W^{1,p}(B)\to W^{1-1/p,p}(\partial B)
\]
is bounded, and by the compact trace embedding on the sphere,
\[
W^{1-1/p,p}(\partial B)\hookrightarrow L^q(\partial B)
\]
is compact for \(1\le q<\frac{p(n-1)}{n-p}\).

Next, by the boundary change-of-variables formula \eqref{eq:weak-cov-inverse}, the operator
\[
(\varphi^{-1}_\partial)^*:L^q(\partial B)\to L^q(\partial\Omega,J_{\varphi^{-1}}^\partial),
\qquad
(\varphi^{-1}_\partial)^*(g)=g\circ\varphi^{-1}_\partial,
\]
is well defined and bounded.

Therefore the composition
\[
(\varphi^{-1}_\partial)^*\circ T_B\circ \varphi^*
\]
defines a bounded operator
\[
T:W^{1,p}(\Omega)\to L^q(\partial\Omega,J_{\varphi^{-1}}^\partial).
\]
Since the middle embedding through \(W^{1-1/p,p}(\partial B)\) is compact, the
operator \(T\) is compact as well.
\end{proof}

\begin{remark}
If, in addition, \(\Omega\) is a bounded Lipschitz domain, then the induced trace
operator \(T\) coincides with the classical trace operator on
\(\Omega\).
Indeed, for every \(u\in C(\overline\Omega)\cap W^{1,p}(\Omega)\), one has
\[
Tu=u|_{\partial\Omega},
\]
while the classical trace operator has the same property. Since both operators are
continuous on \(W^{1,p}(\Omega)\), they coincide on the whole space.
\end{remark}

\subsection{Admissible domains and induced concentrating weights}

We now collect the standing assumptions on the domain \(\Omega\) and on the
transfer mapping \(\varphi\). Unless stated otherwise, these assumptions will
be in force throughout the paper.

\begin{definition}[Admissible transfer mapping]
\label{def:trace-compatible-transfer}
Let \(1<p<n\), let \(\Omega\subset\mathbb R^n\) be a bounded domain with
locally \((n-1)\)-rectifiable boundary, and let
\[
\varphi:\overline B\to\overline\Omega
\]
be a homeomorphism. We say that \(\varphi\) is a
\emph{admissible transfer mapping} if the following
properties hold:
\begin{enumerate}
\item[(A1)] the restriction \(\varphi|_B:B\to\Omega\) is a weak
\(p\)-quasiconformal mapping;

\item[(A2)] \(J_\varphi\in L^r(B)\) for some \(r>\dfrac{n}{p}\);

\item[(A3)] the boundary map
\[
\varphi_\partial:=\varphi|_{\partial B}:\partial B\to\partial\Omega
\]
satisfies the Luzin \(N^{-1}\)-property, i.e.
\[
\mathcal H^{n-1}(E)=0,\quad E\subset\partial\Omega
\qquad\Longrightarrow\qquad
\sigma(\varphi_\partial^{-1}(E))=0 .
\]
\end{enumerate}
\end{definition}

\begin{definition}[Admissible domain]
\label{def:admissible-domain}
We say that a bounded domain \(\Omega\subset\mathbb R^n\) is
\emph{admissible} if the following properties hold:
\begin{enumerate}
\item[(B1)] \(\partial\Omega\) is locally \((n-1)\)-rectifiable;

\item[(B2)] \(\partial\Omega\) satisfies the doubling condition;

\item[(B3)] \(\Omega\) is the image of the unit ball \(B\) under a
admissible transfer mapping.
\end{enumerate}
\end{definition}

The class of admissible domains is fairly broad, but it is not completely
geometric in nature: it is defined through the existence of a transfer map
\(\varphi\) with the properties listed above.

A basic source of examples is given by bilipschitz images of the unit ball.
A second basic family is formed by the standard power-type outward cuspidal
domains. For these domains one has explicit weak
\(p\)-quasiconformal parametrizations \(\varphi:B\to\Omega\) such that
\(J_\varphi\in L^r(B)\) for some \(r>n/p\), and the boundary satisfies the
required rectifiability and doubling assumptions (see, e.g. \cite{GarainGoldshteinUkhlov2025, MenovschikovUkhlov2026, LambertiUkhlov2026}).

On the other hand, several natural classes of domains are excluded.
Fractal-type boundaries, such as the Koch snowflake, are not admissible since
they fail local \((n-1)\)-rectifiability. Multiply connected domains, annuli,
and domains with holes are not admissible because condition \textup{(B3)}
requires \(\overline\Omega\) to be the image of \(\overline B\) under a
homeomorphism.

\begin{remark}
The assumption that \(\varphi\) extends to a homeomorphism
\[
\varphi:\overline B\to\overline\Omega
\]
is stronger than what is actually used in the proofs. In fact, for the boundary
part of the argument it would be enough to have, in addition to the interior
homeomorphism \(\varphi:B\to\Omega\), a bijection
\[
\varphi_\partial:\partial B\to\partial\Omega
\]
with measurable inverse \(\psi_\partial=\varphi_\partial^{-1}\), satisfying the
Luzin \(N^{-1}\)-property and the boundary change-of-variables formula, and
compatible with the interior map in the natural sense that it represents the
boundary values of \(\varphi\).

We nevertheless keep the stronger homeomorphic extension hypothesis,
since it is geometrically natural and allows us to avoid introducing a separate
compatibility condition between the interior and boundary maps.
\end{remark}

We now define the class of concentrating bulk weights and the induced boundary
weight.

\begin{definition}[Admissible weights]
\label{def:admissible-weights}
Let \(\alpha\in L^\infty(\partial B)\), \(\alpha\ge0\), \(\alpha\not\equiv0\),
and let \(\widetilde\alpha\) be its radial extension to \(B\):
\[
\widetilde\alpha(x)=
\begin{cases}
\alpha(x/|x|), & x\in B\setminus\{0\},\\
0, & x=0.
\end{cases}
\]
For \(a\in(0,1]\), define
\[
\rho_a(x):=\frac{n}{a}|x|^{n/a-n},
\qquad
\mu_a(x):=\widetilde\alpha(x)\rho_a(x),
\qquad x\in B.
\]
The corresponding induced weights on \(\partial\Omega\) and \(\Omega\) are
\[
\beta(t):=\alpha(\varphi_\partial^{-1}(t))\,J^\partial_{\varphi^{-1}}(t),
\qquad t\in\partial\Omega,
\]
and
\[
\gamma_a(y):=\widetilde\alpha(\varphi^{-1}(y))\,\rho_a(\varphi^{-1}(y))\,
J_{\varphi^{-1}}(y),
\qquad y\in\Omega.
\]
We call the triple \((\mu_a,\gamma_a,\beta)\) the family of \textit{admissible weights}
associated with \(\alpha\) and \(\varphi\).
\end{definition}

By \eqref{chvf2},
\[
\int_\Omega f(y)\,\gamma_a(y)\,dy
=
\int_B f(\varphi(x))\,\mu_{a}(x)\,dx
\]
for every measurable \(f:\Omega\to\mathbb R\) for which the integrals are well defined.

The total mass of all induced weighted measures coincides with the total mass of
the initial boundary weight \(\alpha\), and by \eqref{chvf2} and \eqref{eq:weak-cov-inverse}
\[
\int_B \mu_{a}\,dx
=
\int_\Omega \gamma_a\,dy
=
\int_{\partial\Omega}\beta\,dS
=
\int_{\partial B}\alpha\,d\sigma
=:A_\alpha>0.
\]
In particular, all induced weights are nontrivial and have finite positive mass.

Since \(\alpha\in L^\infty(\partial B)\) and, for every fixed \(a\in(0,1]\), the radial density
\(\rho_a\) is bounded on \(B\), the induced bulk weight \(\gamma_a\) differs from
\(J_{\varphi^{-1}}\) only by a bounded nonnegative factor. Likewise, the induced boundary
weight \(\beta\) differs from \(J^\partial_{\varphi^{-1}}\) only by the bounded factor
\(\alpha\circ\varphi_\partial^{-1}\). Therefore, by Lemma~\ref{lem:bounded-weight-embedding}, all compactness and continuity
statements established above for the weights \(J_{\varphi^{-1}}\) and \(J^\partial_{\varphi^{-1}}\)
remain valid for \(\gamma_a\) and \(\beta\) as well (for fixed \(a\) in the bulk case). 

\begin{remark}
The assumptions \(\alpha\in L^\infty(\partial B)\) and
\[
\rho_a(x)=\frac{n}{a}\,|x|^{n/a-n}
\]
are chosen for clarity rather than sharpness. The argument extends to more general
boundary weights and concentrating families. In particular, the condition
\(\alpha\in L^\infty(\partial B)\) may be weakened to \(\alpha\in L^r(\partial B)\),
\(r>1\), provided
\[
q\,r'<\frac{p(n-1)}{n-p},
\qquad r'=\frac{r}{r-1},
\]
so that the compact trace embedding on \(\partial B\) still yields compactness in
\(L^q(\partial B,\alpha\,d\sigma)\) via H\"older's inequality.

Similarly, the specific densities \(\rho_a\) may be replaced by other concentrating
densities, or even by a family of finite measures in \(B\), as long as the proof still
has the following three ingredients: concentration to the boundary measure
\(\alpha\,d\sigma\) as \(a\to0\), a boundary-layer estimate of the form
\[
\|w\|_{L^q(B,\mu_a)} \le \varepsilon_a \|\nabla w\|_{L^p(B)},
\qquad \varepsilon_a\to0,
\qquad w\in W^{1,p}_0(B),
\]
and compatibility with the boundary approximation scheme on the ball (in the
present setting, through the radial extension and the Poisson semigroup).
\end{remark}

\section{Variational setting for the weighted Neumann and Steklov problems}

In this section we recall the variational framework for the weighted Neumann and weighted Steklov problems associated with the weights \(\gamma_a\) in \(\Omega\) and \(\beta\) on \(\partial\Omega\). The strong formulations below are written only to indicate the structure of the corresponding Euler--Lagrange equations. In the present paper, the actual objects are weak solutions in \(W^{1,p}(\Omega)\). 

We write
\[
\Delta_p u:=\operatorname{div}(|\nabla u|^{p-2}\nabla u).
\]

Let $\Omega \subset \R^n$ be admissible. The Neumann \((p,q)\)-problem with weight $\gamma_a$ is formally written as
\begin{equation}\label{eq:strong-neumann-pq}
\begin{cases}
-\Delta_pu
=
\lambda\,\|u\|_{L^q(\Omega,\gamma_a)}^{\,p-q}\,
|u|^{q-2}u\,\gamma_a
& \text{in }\Omega,
\\[0.4em]
|\nabla u|^{p-2}\nabla u\cdot \nu = 0
& \text{on }\partial\Omega.
\end{cases}
\end{equation}

The corresponding Steklov \((p,q)\)-problem with weight $\beta$ is formally written as
\begin{equation}\label{eq:strong-steklov-pq}
\begin{cases}
-\Delta_pu=0
& \text{in }\Omega,
\\[0.4em]
|\nabla u|^{p-2}\nabla u\cdot \nu
=
\lambda\,\|Tu\|_{L^q(\partial\Omega,\beta)}^{\,p-q}\,
|Tu|^{q-2}Tu\,\beta
& \text{on }\partial\Omega.
\end{cases}
\end{equation}

\subsection{The first non-trivial Neumann eigenvalue}

A pair \((\lambda,u)\in \mathbb R\times W^{1,p}(\Omega)\), \(u\not\equiv 0\), is called
a weak eigenpair of \eqref{eq:strong-neumann-pq} if
\begin{equation}\label{eq:weak-neumann-pq}
\int_\Omega |\nabla u|^{p-2}\nabla u\cdot \nabla v\,dy
=
\lambda\,\|u\|_{L^q(\Omega,\gamma_a)}^{\,p-q}
\int_\Omega |u|^{q-2}u\,v\,\gamma_a\,dy
\end{equation}
for every \(v\in W^{1,p}(\Omega)\).

Further we define reduced variational classes.
For \(a\in(0,1]\), define
\[
V^N_{p,q}(\gamma_a)
:=
\left\{
u\in W^{1,p}(\Omega)\setminus\{0\}:
\int_\Omega |u|^{q-2}u\,\gamma_a\,dy=0
\right\},
\]
and
\begin{equation}\label{eq:def-LambdaN}
\Lambda^N_{p,q}(\gamma_a)
:=
\inf\left\{
\frac{\displaystyle\int_\Omega |\nabla u|^p\,dy}
{\left(\displaystyle\int_\Omega |u|^q\gamma_a\,dy\right)^{p/q}}
:\,
u\in V^N_{p,q}(\gamma_a)
\right\}.
\end{equation}

The orthogonality condition in the definition of \(V^N_{p,q}(\gamma_a)\)
removes constants and singles out the first nontrivial variational level of the
Neumann problem. Indeed, for
\[
F_u(c):=\int_\Omega |u-c|^q\,\gamma_a\,dy,
\]
the minimizer in \(c\in\mathbb R\) is uniquely characterized by
\[
\int_\Omega |u-c|^{q-2}(u-c)\,\gamma_a\,dy=0,
\]
so for \(u\in V^N_{p,q}(\gamma_a)\) the unique minimizing constant is \(c=0\).

We start with the generalization of the standard Friedrichs inequality.

\begin{proposition}[Friedrichs inequality on the reduced Neumann class]
\label{prop:friedrichs-neumann}
Let \(1<p<n\), \(1<q<\frac{pn}{n-p}\), and $\Omega \subset \R^n$ be admissible.

Then there exists a constant \(C_a>0\) such that
\[
\|u\|_{L^p(\Omega)}\le C_a \|\nabla u\|_{L^p(\Omega)}
\]
for every \(u\in W^{1,p}(\Omega)\) satisfying
\[
\int_\Omega |u|^{q-2}u\,\gamma_a\,dy=0.
\]
\end{proposition}

\begin{proof}
Assume, by contradiction, that the stated inequality is false. Then there exists a
sequence \(\{u_k\}\subset W^{1,p}(\Omega)\) such that
\[
\int_\Omega |u_k|^{q-2}u_k\,\gamma_a\,dy=0,
\qquad
\|u_k\|_{L^p(\Omega)} > k \|\nabla u_k\|_{L^p(\Omega)}.
\]
After normalization we may assume that
\[
\|u_k\|_{L^p(\Omega)}=1,
\qquad
\|\nabla u_k\|_{L^p(\Omega)}\to 0.
\]
Hence \(\{u_k\}\) is bounded in \(W^{1,p}(\Omega)\). By Corolary \ref{cor:compact-Lp} the embedding
\(W^{1,p}(\Omega)\hookrightarrow L^p(\Omega)\) is compact and after passing to a subsequence,
\[
u_k\to u
\qquad\text{strongly in }L^p(\Omega).
\]
Since \(\|\nabla u_k\|_{L^p(\Omega)}\to0\), it follows that \(\nabla u=0\) a.e. in
\(\Omega\), and therefore \(u\equiv c\) is constant, because \(\Omega\) is connected.

On the other hand, by Theorem \ref{thm:bulk-jacobian} the embedding
\[
W^{1,p}(\Omega)\hookrightarrow L^q(\Omega,\gamma_a),
\]
is compact and after passing to a further subsequence if necessary,
\[
u_k\to u
\qquad\text{strongly in }L^q(\Omega,\gamma_a).
\]
Since the Nemytskii map
\[
\Phi_q(t):=|t|^{q-2}t
\]
is continuous from \(L^q(\Omega,\gamma_a)\) to \(L^{q'}(\Omega,\gamma_a)\),
\(q'=\frac{q}{q-1}\), we may pass to the limit in the orthogonality condition:
\[
0=\lim_{k\to\infty}\int_\Omega |u_k|^{q-2}u_k\,\gamma_a\,dy
=
\int_\Omega |u|^{q-2}u\,\gamma_a\,dy.
\]
Since \(u\equiv c\), this gives
\[
|c|^{q-2}c \int_\Omega \gamma_a\,dy=0.
\]
By the assumption \(\int_\Omega \gamma_a\,dy>0\), we conclude that \(c=0\). Hence
\(u\equiv0\).

But this contradicts the strong convergence in \(L^p(\Omega)\), because
\[
1=\|u_k\|_{L^p(\Omega)}\to \|u\|_{L^p(\Omega)}=0.
\]
The contradiction proves the proposition.
\end{proof}

The variational characterization of the first nontrivial Neumann level considered
here fits into a broader nonlinear spectral literature. For the classical
\(p\)-Laplacian with homogeneous Neumann boundary condition, see for example
\cite{Huang1990}. A broader treatment of nonlinear eigenvalue problems
for the \(p\)-Laplacian under several types of boundary conditions was given by
L\^e~\cite{Le2006}. More recent results for \((p,q)\)-type Neumann
problems in singular or irregular geometries, closer to the present framework,
can be found for instance in \cite{GarainPchelintsevUkhlov2024, MenovschikovUkhlov2026}.
The next proposition records the corresponding existence statement for the
weighted problem \eqref{eq:weak-neumann-pq} on admissible domains.

\begin{proposition}[First nontrivial weighted Neumann eigenvalue]\label{prop:first-neumann} 
Let \(1<p<n\), \(1<q<\frac{pn}{n-p}\), and $\Omega \subset \R^n$ be admissible.
Then \( \Lambda^N_{p,q}(\gamma_a)>0\), the infimum in \eqref{eq:def-LambdaN} is attained, and every minimizer \(u_a\in V^N_{p,q}(\gamma_a)\) yields a weak eigenpair \( \bigl(\Lambda^N_{p,q}(\gamma_a),u_a\bigr) \) of \eqref{eq:weak-neumann-pq}. In particular, 
\[ 
\Lambda^N_{p,q}(\gamma_a)=\inf\{\lambda>0: \lambda\,\, \text{ is an eigenvalue of \eqref{eq:weak-neumann-pq}}\}. 
\] 
\end{proposition} 
\begin{proof}
By the weighted \((p,q)\)-Poincar\'e inequality on \((\Omega,\gamma_a)\), for every
\(u\in V^N_{p,q}(\gamma_a)\), due to the orthogonality condition one has
\[
\|u\|_{L^q(\Omega,\gamma_a)}
=
\inf_{c\in\mathbb R}\|u-c\|_{L^q(\Omega,\gamma_a)}
\le C\,\|\nabla u\|_{L^p(\Omega)},
\]
Therefore
\[
\Lambda^N_{p,q}(\gamma_a)
=
\inf_{u\in V^N_{p,q}(\gamma_a)}
\frac{\displaystyle\int_\Omega |\nabla u|^p\,dy}
{\displaystyle\left(\int_\Omega |u|^q\gamma_a\,dy\right)^{p/q}}
\ge C^{-p}>0.
\]

Let \(\{u_k\}\subset V^N_{p,q}(\gamma_a)\) be a minimizing sequence, normalized by
\[
\int_\Omega |u_k|^q\gamma_a\,dy=1.
\]
Then
\[
\int_\Omega |\nabla u_k|^p\,dy\to \Lambda^N_{p,q}(\gamma_a),
\]
so in particular \(\{\nabla u_k\}\) is bounded in \(L^p(\Omega)\). By the Friedrichs
inequality on the reduced Neumann class, there exists \(C_a>0\) such that
\[
\|u_k\|_{L^p(\Omega)}\le C_a\|\nabla u_k\|_{L^p(\Omega)}.
\]
Hence \(\{u_k\}\) is bounded in \(W^{1,p}(\Omega)\).

Passing to a subsequence, we may assume that \(u_k\) converges weakly to \(u_a\) in \(W^{1,p}(\Omega)\).
Since the embeddings
\[
W^{1,p}(\Omega)\hookrightarrow L^p(\Omega)
\qquad\text{and}\qquad
W^{1,p}(\Omega)\hookrightarrow L^q(\Omega,\gamma_a)
\]
are compact, after passing to a further subsequence if necessary we also have
\[
u_k\to u_a \quad\text{strongly in }L^p(\Omega) \quad \text{and}\quad u_k\to u_a \quad\text{strongly in }L^q(\Omega,\gamma_a).
\]
In particular,
\[
\int_\Omega |u_a|^q\gamma_a\,dy=1.
\]
As in the previous proposition, we may pass to the limit in the orthogonality condition:
\[
0=\lim_{k\to\infty}\int_\Omega |u_k|^{q-2}u_k\,\gamma_a\,dy
=
\int_\Omega |u_a|^{q-2}u_a\,\gamma_a\,dy.
\]
Hence \(u_a\in V^N_{p,q}(\gamma_a)\).

By the weak lower semicontinuity of the \(L^p\)-norm of the gradient,
\[
\int_\Omega |\nabla u_a|^p\,dy
\le
\liminf_{k\to\infty}\int_\Omega |\nabla u_k|^p\,dy
=
\Lambda^N_{p,q}(\gamma_a).
\]
Since \(u_a\) is admissible and normalized, the reverse inequality is automatic:
\[
\Lambda^N_{p,q}(\gamma_a)\le \int_\Omega |\nabla u_a|^p\,dy.
\]
Therefore
\[
\int_\Omega |\nabla u_a|^p\,dy=\Lambda^N_{p,q}(\gamma_a),
\]
so \(u_a\) attains the infimum in \eqref{eq:def-LambdaN}.

It remains to show that \(u_a\) yields a weak eigenpair of
\eqref{eq:weak-neumann-pq}. The minimizer \(u_a\) minimizes \(\int_\Omega |\nabla u|^p\,dy\) on the constraint set
\[
 \{u\in W^{1,p}(\Omega): \int_\Omega |u|^q\gamma_a\,dy=1,\ \int_\Omega |u|^{q-2}u\,\gamma_a\,dy=0\}.
\]
A standard constrained-variation argument yields
\[
\int_\Omega |\nabla u_a|^{p-2}\nabla u_a\cdot \nabla v\,dy
=
\lambda \int_\Omega |u_a|^{q-2}u_a\,v\,\gamma_a\,dy
+
\mu \int_\Omega |u_a|^{q-2}v\,\gamma_a\,dy
\]
for all \(v\in W^{1,p}(\Omega)\), with suitable constants \(\lambda,\mu\in\mathbb R\).

Choosing \(v\equiv 1\), we obtain
\[
0
=
\lambda \int_\Omega |u_a|^{q-2}u_a\,\gamma_a\,dy
+
\mu \int_\Omega |u_a|^{q-2}\gamma_a\,dy.
\]
The first integral vanishes because \(u_a\in V^N_{p,q}(\gamma_a)\), while the second
one is strictly positive. Hence \(\mu=0\). Therefore
\[
\int_\Omega |\nabla u_a|^{p-2}\nabla u_a\cdot \nabla v\,dy
=
\lambda \int_\Omega |u_a|^{q-2}u_a\,v\,\gamma_a\,dy
\qquad\text{for all }v\in W^{1,p}(\Omega).
\]
Finally, taking \(v=u_a\) and using the normalization, we obtain
\[
\lambda
=
\int_\Omega |\nabla u_a|^p\,dy
=
\Lambda^N_{p,q}(\gamma_a).
\]
Thus \(\bigl(\Lambda^N_{p,q}(\gamma_a),u_a\bigr)\) is a weak eigenpair of
\eqref{eq:weak-neumann-pq}.

Conversely, let \((\lambda,u)\) be a weak eigenpair of \eqref{eq:weak-neumann-pq}
with \(\lambda>0\). Taking \(v\equiv 1\) in \eqref{eq:weak-neumann-pq}, we get
\[
\int_\Omega |u|^{q-2}u\,\gamma_a\,dy=0,
\]
so \(u\in V^N_{p,q}(\gamma_a)\). Taking \(v=u\) in
\eqref{eq:weak-neumann-pq}, we obtain
\[
\int_\Omega |\nabla u|^p\,dy
=
\lambda \left(\int_\Omega |u|^q\gamma_a\,dy\right)^{p/q},
\]
and therefore
\[
\lambda
=
\frac{\displaystyle\int_\Omega |\nabla u|^p\,dy}
{\displaystyle\left(\int_\Omega |u|^q\gamma_a\,dy\right)^{p/q}}
\ge
\Lambda^N_{p,q}(\gamma_a).
\]
Since the opposite inequality is already realized by the minimizer \(u_a\), we conclude
that
\[
\Lambda^N_{p,q}(\gamma_a)
=
\inf\{\lambda>0:\ \lambda \text{ is an eigenvalue of \eqref{eq:weak-neumann-pq}}\}.
\]
This completes the proof.
\end{proof}

\subsection{The first non-trivial Steklov eigenvalue}

A pair \((\lambda,u)\in \mathbb R\times W^{1,p}(\Omega)\), \(Tu\not\equiv 0\), is called
a weak eigenpair of \eqref{eq:strong-steklov-pq} if
\begin{equation}\label{eq:weak-steklov-pq}
\int_\Omega |\nabla u|^{p-2}\nabla u\cdot \nabla v\,dy
=
\lambda\,\|Tu\|_{L^q(\partial\Omega,\beta)}^{\,p-q}
\int_{\partial\Omega}|Tu|^{q-2}Tu\,Tv\,\beta\,dS
\end{equation}
for every \(v\in W^{1,p}(\Omega)\).

Similarly to Neumann case, define
\[
V^{St}_{p,q}(\beta)
:=
\left\{
u\in W^{1,p}(\Omega):
Tu\not\equiv 0,\ 
\int_{\partial\Omega}|Tu|^{q-2}Tu\,\beta\,dS=0
\right\},
\]
and
\begin{equation}\label{eq:def-LambdaSt}
\Lambda^{St}_{p,q}(\beta)
:=
\inf\left\{
\frac{\displaystyle\int_\Omega |\nabla u|^p\,dy}
{\left(\displaystyle\int_{\partial\Omega}|Tu|^q\beta\,dS\right)^{p/q}}
:\,
u\in V^{St}_{p,q}(\beta)
\right\}.
\end{equation}

\begin{proposition}[Friedrichs inequality on the reduced Steklov class]
\label{prop:friedrichs-steklov}
Let \(1<p<n\), \(1<q<\frac{p(n-1)}{n-p}\), and \(\Omega\subset\mathbb R^n\) be admissible.

Then there exists a constant \(C>0\) such that
\[
\|u\|_{L^p(\Omega)}\le C \|\nabla u\|_{L^p(\Omega)}
\]
for every \(u\in W^{1,p}(\Omega)\) satisfying
\[
\int_{\partial\Omega} |Tu|^{q-2}Tu\,\beta\,dS=0.
\]
\end{proposition}

\begin{proof}
The proof is completely analogous to Proposition~\ref{prop:friedrichs-neumann},
with the compact embedding
\[
W^{1,p}(\Omega)\hookrightarrow L^q(\Omega,\gamma_a)
\]
replaced by the compact trace operator
\[
T:W^{1,p}(\Omega)\to L^q(\partial\Omega,\beta).
\]
\end{proof}

Related nonlinear Steklov-type eigenvalue problems have also been studied from
the viewpoint of trace embeddings and variational methods; see for example
\cite{BonderRossi2002, PaganiPierotti2010, GarainGoldshteinUkhlov2025, LambertiUkhlov2026}. We now record the analogous result on the first nontrivial weighted Steklov eigenvalue introduced above.

\begin{proposition}[First nontrivial weighted Steklov eigenvalue]\label{prop:first-steklov} 
Let \(1<p<n\), \(1<q<\frac{p(n-1)}{n-p}\), and \(\Omega\subset\mathbb R^n\) be admissible.
Then \( \Lambda^{St}_{p,q}(\beta)>0, \) the infimum in \eqref{eq:def-LambdaSt} is attained, and every minimizer \(u\in V^{St}_{p,q}(\beta)\) yields a weak eigenpair \( \bigl(\Lambda^{St}_{p,q}(\beta),u\bigr) \) of \eqref{eq:weak-steklov-pq}. In particular, 
\[ 
\Lambda^{St}_{p,q}(\beta)=\inf\{\lambda>0: \lambda\,\, \text{ is an eigenvalue of \eqref{eq:weak-steklov-pq}}\}. 
\] 
\end{proposition} 

\begin{proof}
The proof is completely analogous to Proposition~\ref{prop:first-neumann}, with the compact embedding $W^{1,p}(\Omega) \hookrightarrow L^q(\Omega, \gamma_a)$ replaced by the compact weighted trace operator $T: W^{1,p}(\Omega) \hookrightarrow L^q(\partial\Omega, \beta)$, and with Proposition~\ref{prop:friedrichs-neumann} replaced by Proposition~\ref{prop:friedrichs-steklov}.
\end{proof}

\subsection{Equivalent formulation in terms of sharp inequalities}
It is convenient to rewrite the variational quantities
\(\Lambda^N_{p,q}(\gamma_a)\) and \(\Lambda^{St}_{p,q}(\beta)\) in terms of the best
constants in the corresponding weighted Poincar\'e and trace inequalities.

For \(a\in(0,1]\), let \(C^N_{p,q}(\gamma_a)\) denote the best constant in the
weighted \((p,q)\)-Poincar\'e inequality
\begin{equation}\label{eq:best-poincare-constant}
\inf_{c\in\mathbb R}\|u-c\|_{L^q(\Omega,\gamma_a)}
\le
C^N_{p,q}(\gamma_a)\,\|\nabla u\|_{L^p(\Omega)},
\qquad u\in W^{1,p}(\Omega).
\end{equation}

Similarly, let \(C^{St}_{p,q}(\beta)\) denote the best constant in the weighted trace
inequality
\begin{equation}\label{eq:best-trace-constant}
\inf_{c\in\mathbb R}\|Tu-c\|_{L^q(\partial\Omega,\beta)}
\le
C^{St}_{p,q}(\beta)\,\|\nabla u\|_{L^p(\Omega)},
\qquad u\in W^{1,p}(\Omega).
\end{equation}

\begin{proposition}[Sharp constants and first nontrivial eigenvalues]
\label{prop:sharp-constants}
Let \(1<p<n\), \(1<q<\frac{p(n-1)}{n-p}\), and \(\Omega\subset\mathbb R^n\) be admissible. Then
\[
C^N_{p,q}(\gamma_a)=\Lambda^N_{p,q}(\gamma_a)^{-1/p},
\qquad
C^{St}_{p,q}(\beta)=\Lambda^{St}_{p,q}(\beta)^{-1/p}.
\]
\end{proposition}

\begin{proof}
We first consider the Neumann side. For a fixed \(u\in W^{1,p}(\Omega)\), define
\[
F_u(c):=\int_\Omega |u-c|^q\,\gamma_a\,dy,
\qquad c\in\mathbb R.
\]
Since \(\gamma_a\ge 0\) and \(\int_\Omega \gamma_a\,dy>0\), the function \(F_u\) is
continuous, strictly convex, and satisfies \(F_u(c)\to\infty\) as \(|c|\to\infty\).
Hence there exists a unique minimizer \(c(u)\in\mathbb R\).
By differentiation with respect to \(c\), this minimizer is characterized by
\[
\int_\Omega |u-c(u)|^{q-2}(u-c(u))\,\gamma_a\,dy=0.
\]
Therefore \(u-c(u)\in V^N_{p,q}(\gamma_a)\), and
\[
\inf_{c\in\mathbb R}\|u-c\|_{L^q(\Omega,\gamma_a)}
=
\|u-c(u)\|_{L^q(\Omega,\gamma_a)}.
\]
Since \(\nabla(u-c(u))=\nabla u\), it follows that
\[
\Lambda^N_{p,q}(\gamma_a)
=
\inf_{\substack{u\in W^{1,p}(\Omega)\\ \nabla u\not\equiv 0}}
\frac{\|\nabla u\|_{L^p(\Omega)}^p}
{\left(\inf_{c\in\mathbb R}\|u-c\|_{L^q(\Omega,\gamma_a)}\right)^p}.
\]
This is equivalent to saying that the best constant in
\eqref{eq:best-poincare-constant} is exactly
\[
C^N_{p,q}(\gamma_a)=\Lambda^N_{p,q}(\gamma_a)^{-1/p}.
\]

The proof for the Steklov side is completely analogous.
\end{proof}

Consequently, the convergence result proved below for the first nontrivial Neumann
and Steklov eigenvalues may be equivalently reformulated as the convergence of the
corresponding sharp constants:
\[
\Lambda^N_{p,q}(\gamma_a)\to \Lambda^{St}_{p,q}(\beta)
\quad\Longleftrightarrow\quad
C^N_{p,q}(\gamma_a)\to C^{St}_{p,q}(\beta),
\qquad a \to 0.
\]
Indeed, by Propositions~\ref{prop:first-neumann} and~\ref{prop:first-steklov} both quantities are strictly positive, and the
function \(t\mapsto t^{-1/p}\) is continuous on \((0,\infty)\).

\section{Boundary concentration}

\subsection{Boundary concentration on the ball}

We begin with the reference case of the unit ball, where the boundary
concentration mechanism can be analyzed explicitly. The goal of this subsection is
to isolate the ingredients that will later be transferred to a general domain
\(\Omega\) by means of the map \(\varphi\): the boundary-layer estimate for
zero-trace functions, the approximation by the radial extension of the trace,
and the convergence of weighted bulk moments to the corresponding boundary
moments.

We start with auxiliary results on the Poisson extension.
Let \(\{\mathcal P_r\}_{0<r<1}\) be the Poisson semigroup on \(\partial B\), i.e.
\[
\mathcal P_r g(s)
:=
\int_{\partial B} P(rs,\eta)\,g(\eta)\,d\sigma,
\qquad s\in\partial B,
\]
where
\[
P(x,\eta)=\frac{1}{\sigma(\partial B)}\frac{1-|x|^2}{|x-\eta|^n}
\]
is the Poisson kernel of the unit ball; see, for example,
\cite[Chapter~1]{ABR01}.

For smooth \(g\in C^\infty(\partial B)\), define the \textit{Poisson extension operator}
$$
\mathcal{P}: C^\infty(\partial B) \to C^\infty(B)
$$
as
\[
\mathcal P[g](x):=\int_{\partial B}P(x,\eta)\,g(\eta)\,d\sigma(\eta),
\quad x\in B; \quad \mathcal P[g](rs):=\mathcal P_r g(s), \quad s\in\partial B.
\]

For \(g\in L^q(\partial B)\), define its radial extension by
\[
\mathcal R[g](rs):=g(s),
\qquad 0<r<1,\ s\in\partial B.
\]

First let us recall the standard fact that the Poisson extension operator is the right inverse to the trace operator.

\begin{lemma}[Poisson extension on the ball]\label{lem:poisson-extension-trace}
Let \(1<p<\infty\). For \(g\in C^\infty(\partial B)\), the function \(\mathcal P[g]\) is harmonic in \(B\), extends continuously to \(\overline B\),
and
\[
T(\mathcal P[g])=g.
\]
Moreover, \(\mathcal P\) extends uniquely to a bounded linear operator
\[
\mathcal P:W^{1-1/p,p}(\partial B)\to W^{1,p}(B)
\]
such that
\[
T(\mathcal P[g])=g
\qquad\text{for all }g\in W^{1-1/p,p}(\partial B).
\]
\end{lemma}

\begin{proof}
For smooth boundary data, the result is the classical Poisson representation for
the Dirichlet problem in the unit ball; see \cite[Theorem~1.17]{ABR01}.

Next, by the trace theorem on the ball, the trace operator
\[
T:W^{1,p}(B)\to W^{1-1/p,p}(\partial B)
\]
is bounded and surjective. Moreover, \cite[Theorem~1.2]{Miyazaki2016} yields a bounded
solution operator for the Dirichlet problem in the ball: for every
\(g\in W^{1-1/p,p}(\partial B)\) there exists a unique harmonic function
\(u\in W^{1,p}(B)\) such that \(Tu=g\), and
\[
\|u\|_{W^{1,p}(B)}\le C\|g\|_{W^{1-1/p,p}(\partial B)}.
\]
Since \(C^\infty(\partial B)\) is dense in \(W^{1-1/p,p}(\partial B)\), this
solution operator extends uniquely from smooth data to all
\(g\in W^{1-1/p,p}(\partial B)\). By uniqueness of the Dirichlet solution, this
extension coincides with the Poisson extension on smooth data.
\end{proof}

The next estimate is the quantitative Poisson-to-radial comparison that will be
combined below with the boundary-layer estimate for zero-trace functions.

\begin{lemma}[Quantitative comparison of Poisson and radial extensions]
\label{lem:poisson-to-radial-in-mu}
Let \(1\le q<\infty\) and \(0<s<1\). Then there exists a constant
\[
C=C(n,q,s,\|\alpha\|_{L^\infty(\partial B)})>0
\]
such that
\[
\|\mathcal P[g]-\mathcal R[g]\|_{L^q(B,\mu_a)}
\le
C a^s\|g\|_{W^{s,q}(\partial B)}
\]
for every \(g\in W^{s,q}(\partial B)\) and every \(a\in(0,1]\).
\end{lemma}

\begin{proof}
By the standard approximation estimate for the Poisson semigroup on the sphere
(see, for example, \cite[Chapters~4]{DaiXu2013}),
\[
\|\mathcal P_r g-g\|_{L^q(\partial B)}
\le
C(n,q,s)(1-r)^s\|g\|_{W^{s,q}(\partial B)}.
\]
Since \(\alpha\in L^\infty(\partial B)\), it follows that
\[
\|\mathcal P_r g-g\|_{L^q(\partial B,\alpha)}
\le
C(n,q,s,\|\alpha\|_{L^\infty})(1-r)^s\|g\|_{W^{s,q}(\partial B)}.
\]

Using polar coordinates, we obtain
\[
\|\mathcal P[g]-\mathcal R[g]\|_{L^q(B,\mu_a)}^q
=
\int_0^1
\|\mathcal P_r g-g\|_{L^q(\partial B,\alpha)}^q
\frac{n}{a}r^{n/a-1}\,dr.
\]
Therefore
\[
\|\mathcal P[g]-\mathcal R[g]\|_{L^q(B,\mu_a)}^q
\le
C\|g\|_{W^{s,q}(\partial B)}^q
\int_0^1
(1-r)^{sq}\frac{n}{a}r^{n/a-1}\,dr.
\]

It remains to estimate the last integral. We have
\[
\int_0^1(1-r)^{sq}\frac{n}{a}r^{n/a-1}\,dr
=
\frac{n}{a}B(sq+1,n/a),
\]
where \(B\) is the Euler beta function. Since
\[
B(sq+1,n/a)
=
\frac{\Gamma(sq+1)\Gamma(n/a)}{\Gamma(sq+1+n/a)},
\]
and
\[
\frac{\Gamma(n/a)}{\Gamma(n/a+sq+1)}
\le
C_{n,q,s}\Bigl(\frac{a}{n}\Bigr)^{sq+1}
\qquad\text{for }0<a\le1,
\]
it follows that
\[
\int_0^1(1-r)^{sq}\frac{n}{a}r^{n/a-1}\,dr
\le
Ca^{sq}.
\]
Hence
\[
\|\mathcal P[g]-\mathcal R[g]\|_{L^q(B,\mu_a)}
\le
C a^s\|g\|_{W^{s,q}(\partial B)}.
\]
\end{proof}

We now establish a quantitative
estimate for functions with zero trace. It shows that, when measured against the
concentrating bulk weights \(\mu_a\), the contribution of \(W^{1,p}_0(B)\)-functions
becomes small as \(a\to0\).

\begin{lemma}[Boundary-layer estimate for zero-trace functions]
\label{lem:zero-trace-boundary-layer}
Let \(1<p<n\), \(1<q<\frac{p(n-1)}{n-p}\).
Let also
\[
\delta_{p,q}
:=
\min\left\{
1-\frac1p,\,
1-\frac{n}{p}+\frac{n-1}{q}
\right\}.
\]
Then there exists a constant
\[
C=C(n,p,q,\|\alpha\|_{L^\infty(\partial B)})>0
\]
such that for every \(a\in(0,1]\) and every \(w\in W^{1,p}_0(B)\),
\[
\|w\|_{L^q(B,\mu_a)}
\le
C\,a^{\delta_{p,q}}\|\nabla w\|_{L^p(B)}.
\]
\end{lemma}

\begin{proof}
Let
\[
d(x):=\dist(x,\partial B)=1-|x|,
\qquad x\in B.
\]
We first record the elementary estimate
\begin{equation}\label{eq:rho-layer-sup}
\sup_{x\in B} d(x)^m\rho_a(x)\le C_m a^{m-1},
\qquad m>0,\ a\in(0,1].
\end{equation}
Indeed, writing \(r=|x|\in(0,1)\), it is enough to estimate
\[
\sup_{0<r<1}(1-r)^m\frac{n}{a}r^{n/a-n}.
\]
Set
\[
f_a(r):=(1-r)^m\frac{n}{a}r^{n/a-n}.
\]
If \(a=1\), then \(f_1(r)=n(1-r)^m\le n\). Let now \(0<a<1\). A direct
calculation shows that the maximum is attained at
\[
r_a=\frac{n/a-n}{n/a-n+m},
\]
and therefore
\[
\sup_{0<r<1}f_a(r)
=
\frac{n}{a}
\left(\frac{m}{n/a-n+m}\right)^m
\left(\frac{n/a-n}{n/a-n+m}\right)^{n/a-n}.
\]
The last factor is bounded by \(1\), while
\[
\frac{m}{n/a-n+m}
=
\frac{ma}{n(1-a)+ma}
\le C_m a.
\]
Hence
\[
\sup_{0<r<1}f_a(r)\le C_m a^{m-1},
\]
which proves \eqref{eq:rho-layer-sup}.

We first treat the case \(q=p\). Since
\[
\mu_a(x)=\alpha_e(x)\rho_a(x),
\]
we have
\[
\int_B |w|^p\,\mu_a\,dx
\le
\|\alpha\|_{L^\infty(\partial B)}
\int_B |w|^p\rho_a\,dx.
\]
Using \eqref{eq:rho-layer-sup} with \(m=p\), we obtain
\[
\rho_a(x)
=
d(x)^{-p}\,d(x)^p\rho_a(x)
\le
C a^{p-1}d(x)^{-p}.
\]
Hence
\[
\int_B |w|^p\,\mu_a\,dx
\le
C a^{p-1}\int_B \frac{|w|^p}{d(x)^p}\,dx.
\]
By the Hardy inequality on the ball,
\[
\int_B \frac{|w|^p}{d(x)^p}\,dx
\le
C\int_B |\nabla w|^p\,dx.
\]
Therefore
\begin{equation}\label{eq:boundary-layer-p}
\|w\|_{L^p(B,\mu_a)}
\le
C a^{(p-1)/p}\|\nabla w\|_{L^p(B)}.
\end{equation}

If \(1<q<p\), then, since
\[
\mu_a(B)=\int_B \mu_a\,dx=\int_{\partial B}\alpha\,d\sigma<\infty,
\]
H\"older's inequality gives
\[
\|w\|_{L^q(B,\mu_a)}
\le
\mu_a(B)^{1/q-1/p}\|w\|_{L^p(B,\mu_a)}.
\]
Combining this with \eqref{eq:boundary-layer-p}, we obtain
\[
\|w\|_{L^q(B,\mu_a)}
\le
C a^{(p-1)/p}\|\nabla w\|_{L^p(B)}.
\]
This proves the claim for \(1<q\le p\).

It remains to consider the case
\[
p<q<\frac{p(n-1)}{n-p}.
\]
Let
\[
p^*:=\frac{np}{n-p}.
\]
Since \(p<q<p^*\), there exists a unique \(\theta\in(0,1)\) such that
\[
q=\theta p+(1-\theta)p^*,
\qquad
\theta=\frac{p^*-q}{p^*-p}.
\]
For a.e. \(x\in B\),
\[
|w(x)|^q
=
\left(\frac{|w(x)|^p}{d(x)^p}\right)^\theta
|w(x)|^{(1-\theta)p^*}d(x)^{p\theta}.
\]
Therefore
\[
\int_B |w|^q\,\mu_a\,dx
\le
\|\alpha\|_{L^\infty(\partial B)}
\int_B
\left(\frac{|w|^p}{d^p}\right)^\theta
|w|^{(1-\theta)p^*}d^{p\theta}\rho_a\,dx.
\]
Applying H\"older's inequality with exponents \(1/\theta\) and \(1/(1-\theta)\),
we get
\[
\int_B |w|^q\,\mu_a\,dx
\le
\|\alpha\|_{L^\infty(\partial B)}
\left(\int_B \frac{|w|^p}{d^p}\,dx\right)^\theta
\left(
\int_B |w|^{p^*}\bigl(d^{p\theta}\rho_a\bigr)^{1/(1-\theta)}\,dx
\right)^{1-\theta}.
\]
Now
\[
\left(
\int_B |w|^{p^*}\bigl(d^{p\theta}\rho_a\bigr)^{1/(1-\theta)}\,dx
\right)^{1-\theta}
\le
\left(\sup_{x\in B}d(x)^{p\theta}\rho_a(x)\right)
\left(\int_B |w|^{p^*}\,dx\right)^{1-\theta}.
\]
Hence, by \eqref{eq:rho-layer-sup} with \(m=p\theta\),
\[
\int_B |w|^q\,\mu_a\,dx
\le
C a^{p\theta-1}
\left(\int_B \frac{|w|^p}{d^p}\,dx\right)^\theta
\left(\int_B |w|^{p^*}\,dx\right)^{1-\theta}.
\]
Using again the Hardy inequality and the Poincar\'e inequality,
we obtain
\[
\int_B |w|^q\,\mu_a\,dx
\le
C a^{p\theta-1}
\|\nabla w\|_{L^p(B)}^{p\theta}
\|\nabla w\|_{L^p(B)}^{(1-\theta)p^*}.
\]
Since
\[
p\theta+(1-\theta)p^*=q,
\]
it follows that
\[
\int_B |w|^q\,\mu_a\,dx
\le
C a^{p\theta-1}\|\nabla w\|_{L^p(B)}^q.
\]
A direct computation shows that
\[
\frac{p\theta-1}{q}
=
\frac{p(n-1)-q(n-p)}{pq}
=
\delta_{p,q}.
\]
This completes the proof.
\end{proof}

We now combine the two auxiliary ingredients established above: the
Poisson-to-radial comparison from Lemma~\ref{lem:poisson-to-radial-in-mu}
and the boundary-layer estimate from
Lemma~\ref{lem:zero-trace-boundary-layer}. This yields a quantitative
approximation of a Sobolev function by the radial extension of its trace with
respect to the concentrating measures \(\mu_a\).

\begin{proposition}[Quantitative approximation by the radial trace extension]
\label{prop:quantitative-ball-trace-approx}
Let \(1<p<n\), \(1<q<\frac{p(n-1)}{n-p}\), and \(\delta_{p,q}\) be as in
Lemma~\ref{lem:zero-trace-boundary-layer}. Then for every
\(s\in(0,\delta_{p,q})\) there exists a constant
\[
C_s=C_s(n,p,q,s,\|\alpha\|_{L^\infty(\partial B)})>0
\]
such that
\[
\|u-\mathcal R[Tu]\|_{L^q(B,\mu_a)}
\le
C_s a^s \|\nabla u\|_{L^p(B)}
\]
for every \(u\in W^{1,p}(B)\) and every \(a\in(0,1]\).
\end{proposition}

\begin{proof}
Let
\[
u_B:=\fint_B u\,dx,
\qquad
u_0:=u-u_B.
\]
Since both \(u-\mathcal R[Tu]\) and \(\nabla u\) are invariant under adding
constants, we have
\[
u-\mathcal R[Tu]=u_0-\mathcal R[Tu_0],
\qquad
\|\nabla u_0\|_{L^p(B)}=\|\nabla u\|_{L^p(B)}.
\]

We decompose
\[
u_0-\mathcal R[Tu_0]
=
\bigl(u_0-\mathcal P[Tu_0]\bigr)
+
\bigl(\mathcal P[Tu_0]-\mathcal R[Tu_0]\bigr).
\]

Since
\[
T\bigl(u_0-\mathcal P[Tu_0]\bigr)=0,
\]
we have
\[
u_0-\mathcal P[Tu_0]\in W^{1,p}_0(B).
\]
Hence Lemma~\ref{lem:zero-trace-boundary-layer} yields
\[
\|u_0-\mathcal P[Tu_0]\|_{L^q(B,\mu_a)}
\le
C a^{\delta_{p,q}}
\|\nabla(u_0-\mathcal P[Tu_0])\|_{L^p(B)}.
\]
By Lemma~\ref{lem:poisson-extension-trace},
\[
\|\mathcal P[Tu_0]\|_{W^{1,p}(B)}
\le
C\|Tu_0\|_{W^{1-1/p,p}(\partial B)}.
\]
Using the boundedness of the trace operator and the Poincar\'e inequality,
we obtain
\[
\|\mathcal P[Tu_0]\|_{W^{1,p}(B)}
\le
C\|u_0\|_{W^{1,p}(B)}
\le
C\|\nabla u\|_{L^p(B)}.
\]
Therefore
\[
\|u_0-\mathcal P[Tu_0]\|_{L^q(B,\mu_a)}
\le
C a^{\delta_{p,q}}\|\nabla u\|_{L^p(B)}.
\]
Since \(0<a\le1\) and \(s<\delta_{p,q}\),
\(a^{\delta_{p,q}}\le a^s\),
and hence
\begin{equation}\label{eq:first-piece-quant-ball}
\|u_0-\mathcal P[Tu_0]\|_{L^q(B,\mu_a)}
\le
C a^s\|\nabla u\|_{L^p(B)}.
\end{equation}

Next, by Lemma~\ref{lem:poisson-to-radial-in-mu},
\[
\|\mathcal P[Tu_0]-\mathcal R[Tu_0]\|_{L^q(B,\mu_a)}
\le
C a^s\|Tu_0\|_{W^{s,q}(\partial B)}.
\]
Since \(s<\delta_{p,q}\), the fractional Sobolev embedding on \(\partial B\)
recalled in Section~2 gives
\[
W^{1-1/p,p}(\partial B)\hookrightarrow W^{s,q}(\partial B).
\]
Therefore
\[
\|Tu_0\|_{W^{s,q}(\partial B)}
\le
C\|Tu_0\|_{W^{1-1/p,p}(\partial B)}
\le
C\|u_0\|_{W^{1,p}(B)}
\le
C\|\nabla u\|_{L^p(B)}.
\]
Thus
\begin{equation}\label{eq:second-piece-quant-ball}
\|\mathcal P[Tu_0]-\mathcal R[Tu_0]\|_{L^q(B,\mu_a)}
\le
C a^s\|\nabla u\|_{L^p(B)}.
\end{equation}

Combining \eqref{eq:first-piece-quant-ball} and
\eqref{eq:second-piece-quant-ball}, we conclude that
\[
\|u-\mathcal R[Tu]\|_{L^q(B,\mu_a)}
=
\|u_0-\mathcal R[Tu_0]\|_{L^q(B,\mu_a)}
\le
C_s a^s\|\nabla u\|_{L^p(B)}.
\]
This proves the proposition.
\end{proof}

From the above proposition immediately follows the next two corollaries.

\begin{corollary}[Qualitative approximation along concentrating measures]
\label{cor:qualitative-ball-trace-approx}
Let \(1<p<n\), \(1<q<\frac{p(n-1)}{n-p}\), \(a_k\to0\), and let \(\{u_k\}\) be bounded in \(W^{1,p}(B)\). Then
\[
u_k-\mathcal R[Tu_k]\to0
\qquad\text{strongly in }L^q(B,\mu_{a_k}).
\]
\end{corollary}

\begin{corollary}[Ball concentration of weighted moments]
\label{cor:ball-moment-convergence}
Let \(1<p<n\), \(1<q<\frac{p(n-1)}{n-p}\), \(a_k\to0\), and let \(\{u_k\}\) be bounded in \(W^{1,p}(B)\). Assume that
there exists \(g\in L^q(\partial B,\alpha)\) such that
\[
Tu_k\to g
\qquad\text{strongly in }L^q(\partial B,\alpha).
\]
Then
\[
\int_B |u_k|^q\,\mu_{a_k}\,dx
\longrightarrow
\int_{\partial B}|g|^q\alpha\,d\sigma,
\]
and
\[
\int_B |u_k|^{q-2}u_k\,\mu_{a_k}\,dx
\longrightarrow
\int_{\partial B}|g|^{q-2}g\,\alpha\,d\sigma.
\]
\end{corollary}

\begin{proof}
Set
\[
h_k:=\mathcal R[Tu_k].
\]
By Corollary~\ref{cor:qualitative-ball-trace-approx},
\[
u_k-h_k\to0
\qquad\text{strongly in }L^q(B,\mu_{a_k}).
\]

Using polar coordinates, we obtain
\[
\int_B |h_k|^q\,\mu_{a_k}\,dx
=
\int_{\partial B}|Tu_k|^q\alpha\,d\sigma
\]
and
\[
\int_B |h_k|^{q-2}h_k\,\mu_{a_k}\,dx
=
\int_{\partial B}|Tu_k|^{q-2}Tu_k\,\alpha\,d\sigma.
\]
Since \(Tu_k\to g\) strongly in \(L^q(\partial B,\alpha)\), it follows that
\[
\int_B |h_k|^q\,\mu_{a_k}\,dx
\longrightarrow
\int_{\partial B}|g|^q\alpha\,d\sigma
\]
and
\[
\int_B |h_k|^{q-2}h_k\,\mu_{a_k}\,dx
\longrightarrow
\int_{\partial B}|g|^{q-2}g\,\alpha\,d\sigma.
\]

It remains to show that the same limits hold with \(u_k\) in place of \(h_k\).

First, by the pointwise inequality
\[
\bigl||s|^q-|t|^q\bigr|
\le
C_q\bigl(|s|^{q-1}+|t|^{q-1}\bigr)|s-t|,
\qquad s,t\in\mathbb R,
\]
and H\"older's inequality, we get
\[
\int_B \bigl||u_k|^q-|h_k|^q\bigr|\,\mu_{a_k}\,dx
\le
C_q\Bigl(
\|u_k\|_{L^q(B,\mu_{a_k})}^{q-1}
+
\|h_k\|_{L^q(B,\mu_{a_k})}^{q-1}
\Bigr)
\|u_k-h_k\|_{L^q(B,\mu_{a_k})}.
\]
Now
\[
\|h_k\|_{L^q(B,\mu_{a_k})}^q
=
\|Tu_k\|_{L^q(\partial B,\alpha)}^q,
\]
hence \(\{h_k\}\) is bounded in \(L^q(B,\mu_{a_k})\). Therefore \(\{u_k\}\) is
also bounded in \(L^q(B,\mu_{a_k})\), since \(u_k-h_k\to0\) in that space.
Consequently,
\[
\int_B \bigl||u_k|^q-|h_k|^q\bigr|\,\mu_{a_k}\,dx\to0,
\]
which proves
\[
\int_B |u_k|^q\,\mu_{a_k}\,dx
\longrightarrow
\int_{\partial B}|g|^q\alpha\,d\sigma.
\]

For the nonlinear moments, set
\[
\Phi_q(t):=|t|^{q-2}t,
\qquad
q':=\frac{q}{q-1}.
\]
We claim that
\[
\|\Phi_q(u_k)-\Phi_q(h_k)\|_{L^{q'}(B,\mu_{a_k})}\to0.
\]
Indeed, if \(q\ge2\), then
\[
|\Phi_q(s)-\Phi_q(t)|
\le
C_q\bigl(|s|^{q-2}+|t|^{q-2}\bigr)|s-t|,
\]
and H\"older's inequality gives
\[
\|\Phi_q(u_k)-\Phi_q(h_k)\|_{L^{q'}(B,\mu_{a_k})}
\le
C_q
\Bigl(
\|u_k\|_{L^q(B,\mu_{a_k})}^{q-2}
+
\|h_k\|_{L^q(B,\mu_{a_k})}^{q-2}
\Bigr)
\|u_k-h_k\|_{L^q(B,\mu_{a_k})}.
\]
If \(1<q<2\), then
\[
|\Phi_q(s)-\Phi_q(t)|
\le
C_q|s-t|^{q-1},
\]
and therefore
\[
\|\Phi_q(u_k)-\Phi_q(h_k)\|_{L^{q'}(B,\mu_{a_k})}^{q'}
\le
C_q\int_B |u_k-h_k|^q\,\mu_{a_k}\,dx.
\]
In both cases we obtain
\[
\|\Phi_q(u_k)-\Phi_q(h_k)\|_{L^{q'}(B,\mu_{a_k})}\to0.
\]

Since
\[
\mu_{a_k}(B)=\int_{\partial B}\alpha\,d\sigma
\]
is finite and independent of \(k\), H\"older's inequality implies
\[
\left|
\int_B \Phi_q(u_k)\,\mu_{a_k}\,dx
-
\int_B \Phi_q(h_k)\,\mu_{a_k}\,dx
\right|
\le
\mu_{a_k}(B)^{1/q}
\|\Phi_q(u_k)-\Phi_q(h_k)\|_{L^{q'}(B,\mu_{a_k})}
\to0.
\]
Hence
\[
\int_B |u_k|^{q-2}u_k\,\mu_{a_k}\,dx
\longrightarrow
\int_{\partial B}|g|^{q-2}g\,\alpha\,d\sigma.
\]
\end{proof}

\subsection{Transfer from the ball to admissible domains}

In the previous subsection we established the concentration mechanism on
the unit ball. We now transfer these results to an admissible domain
\(\Omega\) through the change-of-variables scheme induced by the transfer map
\(\varphi\).

Throughout this subsection we assume that \(\Omega\subset\mathbb R^n\) is admissible, i.e. satisfies
\((B1)\)--\((B3)\), and we fix an admissible transfer mapping
\[
\varphi:\overline B\to\overline\Omega.
\]

\begin{lemma}[Weighted boundary concentration on \(\Omega\)]
\label{lem:omega-concentration}
Let \(1<p<n\), \(1<q<\frac{p(n-1)}{n-p}\), let \(a_k\to0\), and let
\(\{u_k\}\) be bounded in \(W^{1,p}(\Omega)\). Assume that there exists
\(g\in L^q(\partial\Omega,\beta)\) such that
\[
Tu_k\to g
\qquad\text{strongly in }L^q(\partial\Omega,\beta).
\]
Then
\[
\int_\Omega |u_k|^q\,\gamma_{a_k}\,dy
\longrightarrow
\int_{\partial\Omega}|g|^q\,\beta\,dS,
\]
and
\[
\int_\Omega |u_k|^{q-2}u_k\,\gamma_{a_k}\,dy
\longrightarrow
\int_{\partial\Omega}|g|^{q-2}g\,\beta\,dS.
\]
\end{lemma}

\begin{proof}
Set
\[
v_k:=u_k\circ\varphi.
\]
Since \(\varphi^\ast:u\mapsto u\circ\varphi\) is a bounded composition operator
from \(W^{1,p}(\Omega)\) to \(W^{1,p}(B)\), the sequence \(\{v_k\}\) is bounded
in \(W^{1,p}(B)\).

By the definition of the induced trace operator and the boundary
change-of-variables formula,
\[
T_Bv_k=(Tu_k)\circ\varphi_\partial
\qquad\text{in }L^q(\partial B,\alpha),
\]
where \(T_B\) denotes the classical trace operator on \(B\). Hence
\[
T_Bv_k\to g\circ\varphi_\partial
\qquad\text{strongly in }L^q(\partial B,\alpha).
\]

Applying Corollary~\ref{cor:ball-moment-convergence} to the sequence
\(\{v_k\}\), we obtain
\[
\int_B |v_k|^q\,\mu_{a_k}\,dx
\longrightarrow
\int_{\partial B}|g(\varphi_\partial(s))|^q\,\alpha(s)\,d\sigma(s),
\]
and
\[
\int_B |v_k|^{q-2}v_k\,\mu_{a_k}\,dx
\longrightarrow
\int_{\partial B}|g(\varphi_\partial(s))|^{q-2}g(\varphi_\partial(s))\,
\alpha(s)\,d\sigma(s).
\]

Using the change-of-variables formulae, we
rewrite the left-hand sides as
\[
\int_B |v_k|^q\,\mu_{a_k}\,dx
=
\int_\Omega |u_k|^q\,\gamma_{a_k}\,dy,
\qquad
\int_B |v_k|^{q-2}v_k\,\mu_{a_k}\,dx
=
\int_\Omega |u_k|^{q-2}u_k\,\gamma_{a_k}\,dy,
\]
and the right-hand sides as
\[
\int_{\partial B}|g(\varphi_\partial(s))|^q\,\alpha(s)\,d\sigma(s)
=
\int_{\partial\Omega}|g|^q\,\beta\,dS,
\]
\[
\int_{\partial B}|g(\varphi_\partial(s))|^{q-2}g(\varphi_\partial(s))\,
\alpha(s)\,d\sigma(s)
=
\int_{\partial\Omega}|g|^{q-2}g\,\beta\,dS.
\]
This proves the claim.
\end{proof}

The quantitative approximation obtained on the ball also transfers to
\(\Omega\). As a consequence, one can compare the bulk and boundary quotient
seminorms that define the sharp weighted Poincar\'e and trace constants.

\begin{lemma}[Quantitative comparison of quotient seminorms]
\label{lem:quotient-seminorm-comparison}
Let \(1<p<n\), \(1<q<\frac{p(n-1)}{n-p}\), and \(\Omega\subset\mathbb R^n\) be admissible. Let also
\[
\delta_{p,q}
:=
\min\left\{
1-\frac1p,\,
1-\frac{n}{p}+\frac{n-1}{q}
\right\}.
\]
Then for every \(s\in(0,\delta_{p,q})\) there exists a constant \(C_s>0\) such
that for every \(a\in(0,1]\) and every \(u\in W^{1,p}(\Omega)\),
\[
\left|
\inf_{c\in\mathbb R}\|u-c\|_{L^q(\Omega,\gamma_a)}
-
\inf_{c\in\mathbb R}\|Tu-c\|_{L^q(\partial\Omega,\beta)}
\right|
\le
C_s a^s\|\nabla u\|_{L^p(\Omega)}.
\]
\end{lemma}

\begin{proof}
Set
\[
v:=u\circ\varphi.
\]
By the bulk change-of-variables formula,
\[
\|u-c\|_{L^q(\Omega,\gamma_a)}
=
\|v-c\|_{L^q(B,\mu_a)}
\qquad\text{for every }c\in\mathbb R.
\]
Similarly, by the definition of the induced trace operator and the boundary
change-of-variables formula,
\[
\|Tu-c\|_{L^q(\partial\Omega,\beta)}
=
\|T_Bv-c\|_{L^q(\partial B,\alpha)},
\]
where \(T_B\) denotes the classical trace operator on \(B\).

By Proposition~\ref{prop:quantitative-ball-trace-approx}, for every
\(s\in(0,\delta_{p,q})\),
\[
\|v-\mathcal R[T_Bv]\|_{L^q(B,\mu_a)}
\le
C_s a^s\|\nabla v\|_{L^p(B)}.
\]
Since \(\varphi\) is admissible, we have
\[
\|\nabla v\|_{L^p(B)}
=
\|\nabla(u\circ\varphi)\|_{L^p(B)}
\le
C_\varphi\|\nabla u\|_{L^p(\Omega)}.
\]
Hence
\begin{equation}\label{eq:quotient-proof-basic-approx}
\|v-\mathcal R[T_Bv]\|_{L^q(B,\mu_a)}
\le
C_s a^s\|\nabla u\|_{L^p(\Omega)}.
\end{equation}

Let \(c\in\mathbb R\). Since
\[
\mathcal R[T_Bv-c]=\mathcal R[T_Bv]-c,
\]
the triangle inequality gives
\[
\|v-c\|_{L^q(B,\mu_a)}
\le
\|\mathcal R[T_Bv-c]\|_{L^q(B,\mu_a)}
+
\|v-\mathcal R[T_Bv]\|_{L^q(B,\mu_a)}.
\]
By the definition of \(\mu_a\),
\[
\|\mathcal R[T_Bv-c]\|_{L^q(B,\mu_a)}
=
\|T_Bv-c\|_{L^q(\partial B,\alpha)},
\]
because the radial density \(\rho_a\) has total mass \(1\) in the radial
variable. Therefore, by \eqref{eq:quotient-proof-basic-approx},
\[
\|v-c\|_{L^q(B,\mu_a)}
\le
\|T_Bv-c\|_{L^q(\partial B,\alpha)}
+
C_s a^s\|\nabla u\|_{L^p(\Omega)}.
\]
Taking the infimum over \(c\in\mathbb R\), we obtain
\[
\inf_{c\in\mathbb R}\|v-c\|_{L^q(B,\mu_a)}
\le
\inf_{c\in\mathbb R}\|T_Bv-c\|_{L^q(\partial B,\alpha)}+C_s a^s\|\nabla u\|_{L^p(\Omega)}.
\]

Conversely, for every \(c\in\mathbb R\),
\[
\|T_Bv-c\|_{L^q(\partial B,\alpha)}
=
\|\mathcal R[T_Bv-c]\|_{L^q(B,\mu_a)}.
\]
Applying the triangle inequality once again, we get
\[
\|\mathcal R[T_Bv-c]\|_{L^q(B,\mu_a)}
\le
\|v-c\|_{L^q(B,\mu_a)}
+
\|v-\mathcal R[T_Bv]\|_{L^q(B,\mu_a)}.
\]
Hence, by \eqref{eq:quotient-proof-basic-approx},
\[
\|T_Bv-c\|_{L^q(\partial B,\alpha)}
\le
\|v-c\|_{L^q(B,\mu_a)}
+
C_s a^s\|\nabla u\|_{L^p(\Omega)}.
\]
Taking the infimum over \(c\in\mathbb R\), we arrive at
\[
\inf_{c\in\mathbb R}\|T_Bv-c\|_{L^q(\partial B,\alpha)}
\le
\inf_{c\in\mathbb R}\|v-c\|_{L^q(B,\mu_a)}+C_s a^s\|\nabla u\|_{L^p(\Omega)}.
\]

Combining the two inequalities and returning to \(u\) by the change-of-variables formulae, we obtain
\[
\left|
\inf_{c\in\mathbb R}\|u-c\|_{L^q(\Omega,\gamma_a)}
-
\inf_{c\in\mathbb R}\|Tu-c\|_{L^q(\partial\Omega,\beta)}
\right|
\le
C_s a^s\|\nabla u\|_{L^p(\Omega)}.
\]
This completes the proof.
\end{proof}

\section{Main limit theorem and corollaries}

\begin{theorem}[Quantitative Neumann-to-Steklov convergence]
\label{thm:quantitative-neumann-to-steklov}
Let \(1<p<n\), \(1<q<\frac{p(n-1)}{n-p}\), and let
\(\Omega\subset\mathbb R^n\) be admissible. Let
\[
\delta_{p,q}
:=
\min\left\{
1-\frac1p,\,
1-\frac{n}{p}+\frac{n-1}{q}
\right\}.
\]
Then for every \(s\in(0,\delta_{p,q})\) there exists a constant \(C_s>0\) such that
\[
\left|
C^N_{p,q}(\gamma_a)-C^{St}_{p,q}(\beta)
\right|
\le
C_s a^s
\qquad\text{for every }a\in(0,1].
\]
Consequently,
\[
C^N_{p,q}(\gamma_a)\to C^{St}_{p,q}(\beta)
\qquad\text{as }a\to0,
\]
and, for all sufficiently small \(a>0\),
\[
\left|
\Lambda^N_{p,q}(\gamma_a)-\Lambda^{St}_{p,q}(\beta)
\right|
\le
C_s^\Lambda a^s,
\]
where
\[
C_s^\Lambda
:=
p\,2^{p+1}\bigl(C^{St}_{p,q}(\beta)\bigr)^{-p-1} C_s.
\]
In particular,
\[
\Lambda^N_{p,q}(\gamma_a)\to \Lambda^{St}_{p,q}(\beta)
\qquad\text{as }a\to0.
\]
\end{theorem}

\begin{proof}
By Lemma~\ref{lem:quotient-seminorm-comparison}, for every
\(s\in(0,\delta_{p,q})\) and for every \(u\in W^{1,p}(\Omega)\),
\[
\inf_{c\in\mathbb R}\|u-c\|_{L^q(\Omega,\gamma_a)}\le \inf_{c\in\mathbb R}\|Tu-c\|_{L^q(\partial\Omega,\beta)}+C_s a^s\|\nabla u\|_{L^p(\Omega)},
\]
and
\[
\inf_{c\in\mathbb R}\|Tu-c\|_{L^q(\partial\Omega,\beta)}\le \inf_{c\in\mathbb R}\|u-c\|_{L^q(\Omega,\gamma_a)}+C_s a^s\|\nabla u\|_{L^p(\Omega)}.
\]
Dividing by \(\|\nabla u\|_{L^p(\Omega)}\) and taking the supremum over all
\(u\in W^{1,p}(\Omega)\) with \(\nabla u\not\equiv0\), we obtain
\[
C^N_{p,q}(\gamma_a)\le C^{St}_{p,q}(\beta)+C_s a^s
\]
and
\[
C^{St}_{p,q}(\beta)\le C^N_{p,q}(\gamma_a)+C_s a^s.
\]
Therefore
\[
\left|
C^N_{p,q}(\gamma_a)-C^{St}_{p,q}(\beta)
\right|
\le
C_s a^s
\qquad\text{for every }a\in(0,1].
\]
In particular,
\[
C^N_{p,q}(\gamma_a)\to C^{St}_{p,q}(\beta)
\qquad\text{as }a\to0.
\]

It remains to pass from the sharp constants to the variational levels. By
Proposition~\ref{prop:sharp-constants},
\[
\Lambda^N_{p,q}(\gamma_a)=\bigl(C^N_{p,q}(\gamma_a)\bigr)^{-p},
\qquad
\Lambda^{St}_{p,q}(\beta)=\bigl(C^{St}_{p,q}(\beta)\bigr)^{-p}.
\]
Since \(C^{St}_{p,q}(\beta)>0\) and
\(C^N_{p,q}(\gamma_a)\to C^{St}_{p,q}(\beta)\), there exists \(a_0\in(0,1]\)
such that
\[
C^N_{p,q}(\gamma_a)\ge \frac12\,C^{St}_{p,q}(\beta)
\qquad\text{for all }a\in(0,a_0].
\]
The function \(t\mapsto t^{-p}\) is Lipschitz on \([\frac12\,C^{St}_{p,q}(\beta),\infty)\)
Therefore, for all sufficiently small \(a>0\),
\[
\left|
\Lambda^N_{p,q}(\gamma_a)-\Lambda^{St}_{p,q}(\beta)
\right|
\le
p\,2^{p+1}\bigl(C^{St}_{p,q}(\beta)\bigr)^{-p-1}
\left|
C^N_{p,q}(\gamma_a)-C^{St}_{p,q}(\beta)
\right|.
\]
Combining this with the already proved estimate for the sharp constants, we complete the proof.
\end{proof}

\begin{theorem}[Compactness and convergence of minimizers]
\label{thm:convergence-of-minimizers}
Let \(1<p<n\), \(1<q<\frac{p(n-1)}{n-p}\), and let
\(\Omega\subset\mathbb R^n\) be admissible. Let \(a_k\to0\), and let
\(u_k\) be minimizers for \(\Lambda^N_{p,q}(\gamma_{a_k})\), normalized by
\[
\int_\Omega |u_k|^q\,\gamma_{a_k}\,dy=1.
\]
Then, after passing to a subsequence,
\[
u_k\to u
\qquad\text{strongly in }W^{1,p}(\Omega),
\]
where \(u\) is a minimizer for \(\Lambda^{St}_{p,q}(\beta)\), normalized by
\[
\int_{\partial\Omega}|Tu|^q\,\beta\,dS=1.
\]
\end{theorem}

\begin{proof}
By the normalization and the variational characterization of the minimizers,
\[
\int_\Omega |\nabla u_k|^p\,dy
=
\Lambda^N_{p,q}(\gamma_{a_k})
\qquad\text{for all }k.
\]
By Theorem~\ref{thm:quantitative-neumann-to-steklov},
\[
\Lambda^N_{p,q}(\gamma_{a_k})\to \Lambda^{St}_{p,q}(\beta),
\]
hence \(\{\nabla u_k\}\) is bounded in \(L^p(\Omega)\).

We claim that \(\{u_k\}\) is bounded in \(W^{1,p}(\Omega)\). Set
\[
m_k:=\fint_\Omega u_k\,dy,
\qquad
v_k:=u_k-m_k.
\]
Then
\[
\fint_\Omega v_k\,dy=0,
\qquad
\|\nabla v_k\|_{L^p(\Omega)}
=
\|\nabla u_k\|_{L^p(\Omega)}.
\]
By Corollary \ref{cor:compact-Lp}, we have Poincar\'e inequality on \(\Omega\),
\[
\|v_k\|_{L^p(\Omega)}
\le
C(\Omega) \|\nabla v_k\|_{L^p(\Omega)},
\]
so \(\{v_k\}\) is bounded in \(W^{1,p}(\Omega)\).

It remains to show that \(\{m_k\}\) is bounded. Suppose, to the contrary, that
\(|m_k|\to\infty\) along a subsequence. Define
\[
z_k:=\frac{u_k}{|m_k|}
=
\frac{m_k}{|m_k|}+\frac{v_k}{|m_k|}.
\]
Since \(\{v_k\}\) is bounded in \(W^{1,p}(\Omega)\), after passing to a further
subsequence we may assume that
\[
\frac{m_k}{|m_k|}\to s\in\{-1,1\},
\]
and therefore
\[
z_k\to s
\qquad\text{strongly in }W^{1,p}(\Omega).
\]
By continuity of the induced trace operator,
\[
Tz_k\to s
\qquad\text{strongly in }L^q(\partial\Omega,\beta).
\]
Applying Lemma~\ref{lem:omega-concentration} to the bounded sequence \(\{z_k\}\),
we obtain
\[
\int_\Omega |z_k|^q\,\gamma_{a_k}\,dy
\longrightarrow
\int_{\partial\Omega}\beta\,dS
>0.
\]
On the other hand,
\[
\int_\Omega |z_k|^q\,\gamma_{a_k}\,dy
=
\frac{1}{|m_k|^q}
\int_\Omega |u_k|^q\,\gamma_{a_k}\,dy
=
\frac{1}{|m_k|^q}\to0,
\]
a contradiction. Hence \(\{m_k\}\) is bounded, and therefore \(\{u_k\}\) is
bounded in \(W^{1,p}(\Omega)\).

Passing to a subsequence, we may assume that
\[
u_k\rightharpoonup u
\qquad\text{weakly in }W^{1,p}(\Omega).
\]
By Corollary~\ref{cor:compact-Lp},
\[
u_k\to u
\qquad\text{strongly in }L^p(\Omega),
\]
and by Theorem~\ref{thm:induced-trace-operator}, after passing to a further
subsequence if necessary,
\[
Tu_k\to Tu
\qquad\text{strongly in }L^q(\partial\Omega,\beta).
\]

Since \(u_k\in V^N_{p,q}(\gamma_{a_k})\), we have
\[
\int_\Omega |u_k|^{q-2}u_k\,\gamma_{a_k}\,dy=0
\qquad\text{for all }k.
\]
Applying Lemma~\ref{lem:omega-concentration} to \(\{u_k\}\), we obtain
\[
\int_{\partial\Omega}|Tu|^q\,\beta\,dS
=
\lim_{k\to\infty}\int_\Omega |u_k|^q\,\gamma_{a_k}\,dy
=
1,
\]
and
\[
\int_{\partial\Omega}|Tu|^{q-2}Tu\,\beta\,dS
=
\lim_{k\to\infty}\int_\Omega |u_k|^{q-2}u_k\,\gamma_{a_k}\,dy
=
0.
\]
Thus \(u\in V^{St}_{p,q}(\beta)\), and \(u\) is normalized by
\[
\int_{\partial\Omega}|Tu|^q\,\beta\,dS=1.
\]

By weak lower semicontinuity,
\[
\int_\Omega |\nabla u|^p\,dy
\le
\liminf_{k\to\infty}\int_\Omega |\nabla u_k|^p\,dy
=
\liminf_{k\to\infty}\Lambda^N_{p,q}(\gamma_{a_k})
=
\Lambda^{St}_{p,q}(\beta),
\]
where in the last step we used
Theorem~\ref{thm:quantitative-neumann-to-steklov}. Since \(u\in V^{St}_{p,q}(\beta)\)
and
\[
\int_{\partial\Omega}|Tu|^q\,\beta\,dS=1,
\]
the variational definition of \(\Lambda^{St}_{p,q}(\beta)\) yields
\[
\Lambda^{St}_{p,q}(\beta)
\le
\int_\Omega |\nabla u|^p\,dy.
\]
Therefore
\[
\int_\Omega |\nabla u|^p\,dy
=
\Lambda^{St}_{p,q}(\beta),
\]
so \(u\) is a minimizer for \(\Lambda^{St}_{p,q}(\beta)\).

Finally,
\[
\int_\Omega |\nabla u_k|^p\,dy
=
\Lambda^N_{p,q}(\gamma_{a_k})
\to
\Lambda^{St}_{p,q}(\beta)
=
\int_\Omega |\nabla u|^p\,dy.
\]
Since \(\nabla u_k\rightharpoonup \nabla u\) weakly in \(L^p(\Omega)\), the
uniform convexity of \(L^p(\Omega)\) implies
\[
\nabla u_k\to \nabla u
\qquad\text{strongly in }L^p(\Omega).
\]
Together with the already established strong convergence
\[
u_k\to u
\qquad\text{in }L^p(\Omega),
\]
this gives
\[
u_k\to u
\qquad\text{strongly in }W^{1,p}(\Omega).
\]
This completes the proof.
\end{proof}

\begin{corollary}[Unweighted boundary limit]
\label{cor:unweighted-boundary-limit}
Let \(1<p<n\), \(1<q<\frac{p(n-1)}{n-p}\), and let
\(\Omega\subset\mathbb R^n\) be admissible. Assume that the boundary volume
derivative of the inverse boundary mapping satisfies
\[
0<c\le J^\partial_{\varphi^{-1}}(t)\le C<\infty
\qquad\text{for }H^{n-1}\text{-a.e. }t\in\partial\Omega.
\]
Define
\[
\alpha(s):=\frac{1}{J^\partial_{\varphi^{-1}}(\varphi_\partial(s))},
\qquad s\in\partial B,
\]
Then
\[
C^N_{p,q}(\gamma_a)\to C^{St}_{p,q}(1),
\qquad
\Lambda^N_{p,q}(\gamma_a)\to \Lambda^{St}_{p,q}(1)
\qquad\text{as }a\to0.
\]
Moreover, if \(a_k\to0\) and \(u_k\in V^N_{p,q}(\gamma_{a_k})\) are normalized minimizers for
\(\Lambda^N_{p,q}(\gamma_{a_k})\) then, after passing to a subsequence,
\[
u_k\to u
\qquad\text{strongly in }W^{1,p}(\Omega),
\]
where \(u\in V^{St}_{p,q}(1)\) is a normalized minimizer for \(\Lambda^{St}_{p,q}(1)\).
\end{corollary}

\begin{proof}
By the two-sided bound on \(J^\partial_{\varphi^{-1}}\), the function
\(\alpha\) belongs to \(L^\infty(\partial B)\) and is nonnegative and
nontrivial. By the definition of the induced boundary weight,
\[
\beta(t)
=
\alpha(\varphi_\partial^{-1}(t))\,J^\partial_{\varphi^{-1}}(t).
\]
Substituting the above choice of \(\alpha\), we obtain
\[
\beta(t)
=
\frac{1}{J^\partial_{\varphi^{-1}}(\varphi_\partial(\varphi_\partial^{-1}(t)))}
\,J^\partial_{\varphi^{-1}}(t)
=
1
\]
for \(H^{n-1}\)-a.e. \(t\in\partial\Omega\).

The stated convergence of the sharp constants and eigenvalues now follows from
Theorem~\ref{thm:quantitative-neumann-to-steklov}, while the convergence of
minimizers follows from Theorem~\ref{thm:convergence-of-minimizers}.
\end{proof}

\paragraph{Acknowledgement.}
The author is grateful to Prof. Alexander Ukhlov for introducing him to spectral theory and for his lasting influence on the author's mathematical development.

\bibliographystyle{amsplain}
\bibliography{Neumann_to_Steklov}

\vskip 0.3cm

Alexander Menovschikov; Department of Mathematics, HSE University, Moscow, Russia
 
\emph{E-mail address:} \email{menovschikovmath@gmail.com}

\end{document}